\documentclass[10pt,a4paper]{article}
\usepackage[utf8]{inputenc}
\usepackage{amsmath}
\usepackage{amsfonts}
\usepackage{amssymb}
\usepackage{amsthm}
\usepackage{enumitem}
\usepackage{color}
\usepackage{hyperref}
\usepackage{MnSymbol}
\usepackage{graphicx}

\usepackage[a4paper, textwidth=13.6cm]{geometry}

\newcommand{\R}{\mathbb{R}}
\newcommand{\Z}{\mathbb{Z}}
\newcommand{\N}{\mathbb{N}}
\renewcommand{\oe}{\Omega_{\epsilon}}
\newcommand{\geps}{\Gamma_{\epsilon}}

\newcommand{\veps}{v_{\epsilon}}

\newcommand{\ueps}{u_{\epsilon}}
\newcommand{\peps}{p_{\epsilon}}

\newcommand{\oem}{\Omega_{\epsilon}^M}
\newcommand{\oepm}{\Omega_{\epsilon}^{\pm}}

\newcommand{\weps}{w_{\epsilon}}

\newcommand{\oef}{\Omega_{\epsilon}^f}
\newcommand{\feps}{f_{\epsilon}}
\newcommand{\seps}{S_{\epsilon}}

\newcommand{\sepm}{S_{\epsilon}^{\pm}}
\newcommand{\x}{\bar{x}}
\newcommand{\fxe}{\dfrac{x}{\epsilon}}

\newcommand{\oems}{\Omega_{\epsilon}^{M,s}}
\newcommand{\oemf}{\Omega_{\epsilon}^{M,f}}

\newcommand{\foe}{\dfrac{1}{\epsilon}}
\newcommand{\app}{\mathrm{app}}

\newcommand{\phieps}{\phi_{\epsilon}}
\newcommand{\tphieps}{\widetilde{\phi}_{\epsilon}}
\newcommand{\rats}{\overset{t.s.}{\longrightarrow}}
\newcommand{\tbxfxe}{\left(t,\bar{x},\dfrac{x}{\epsilon}\right)}
\newcommand{\per}{\mathrm{per}}
\newcommand{\spaceH}{\mathcal{H}}
\newcommand{\y}{\bar{y}}

\newcommand{\tveps}{\tilde{v}_{\epsilon}}

\newcommand{\gr}{>}
\newcommand{\kl}{<}

\newtheorem{theorem}{Theorem}
\newtheorem{remark}[theorem]{Remark}
\newtheorem{definition}[theorem]{Definition}
\newtheorem{lemma}[theorem]{Lemma}
\newtheorem{proposition}[theorem]{Proposition}
\newtheorem{corollary}{Corollary}

\author{M. Gahn\thanks{Interdisciplinary Center for Scientific Computing, University of Heidelberg,   Heidelberg, Germany. \textit{Mail: markus.gahn@iwr.uni-heidelberg.de}}
\and W. J\"ager\thanks{Interdisciplinary Center for Scientific Computing, University of Heidelberg,   Heidelberg, Germany. \textit{Mail: wjaeger@iwr.uni-heidelberg.de}}
 \and  M. Neuss-Radu\thanks{Department of Mathematics, Friedrich-Alexander-Universit\"at Erlangen-N\"urnberg,  Erlangen, Germany. \textit{Mail: maria.neuss-radu@math.fau.de }} }
\date{}

\title{Derivation of Stokes-Plate-Equations modeling fluid flow interaction with thin porous elastic layers
}

\begin{document}


\maketitle

\noindent\textit{This paper is dedicated to the memory of Andro Mikeli\'c, an outstanding mathematician, excellent scientific partner and close friend.}
\\

\begin{abstract}

In this paper we investigate the interaction of fluid flow  with a thin porous elastic layer. We consider two fluid-filled bulk domains which are separated by a thin  
periodically perforated  layer consisting of a fluid and an elastic solid part. Thickness and  periodicity of the layer are  of order $\epsilon$, where  $\epsilon$ is small compared to the size of the bulk domains. The fluid flow is described by an instationary Stokes equation  and the solid via linear elasticity. 
The main contribution of this paper is the rigorous homogenization of the porous structure in the layer and the reduction of the layer to an interface $\Sigma$ in the limit $\epsilon \to 0$ using  two-scale convergence.

The effective model consists of the Stokes equation coupled to a time dependent plate equation on the interface $\Sigma$ including homogenized elasticity coefficients carrying information about the micro structure of the layer. In the zeroth order approximation we obtain continuity of the velocities at the interface, where  only a vertical movement occurs and the tangential components vanish. The tangential movement in the solid is of order $\epsilon$ and given as a Kirchhoff-Love displacement. Additionally, we derive higher order correctors for the fluid in the thin layer. 

\end{abstract}

\\


\section{Introduction}

Mathematical modeling of fluid-structure interactions, analysis and numerical simulations of the model systems, their calibrations and validation based on real data are topical in  mathematical and computational research, the results of which are urgently needed and applied in many areas. Knowledge and  data about the structures and the processes on the different scales have grown enormously. Mathematical modeling has to include them properly. This leads to multi-scale systems, which have to be reduced without loss of essential factors. In general, performing scale limits has become a mathematically validated method to reduce complex systems to effective equations, and to replace purely phenomenological approaches by rigorous derivations.
The interaction between dynamics of incompressible Navier-Stokes fluids and poro-elastic structures are of particular interest. In real systems, they also involve diffusion, transport and reaction of chemical or biological species. They may also be coupled with growth of the solid structure,
reaction products may change parameters in the mechanical models, and stresses influence the growth. Important examples are epithelial layers in organisms, controlling the transitions between  compartments, or endothelial layers in blood vessels, separating the lumen and the intima, the inner layer of vessel wall. These transition regions are mainly thin of scale $\epsilon$. Reducing the layer to an interface by passing to the scale limit $\epsilon \to 0$, may make the problem analytically and computationally simpler.
\\
\textit{Andro Mikelic}, to whose memory this paper is dedicated, was one of the pioneers in the analysis of multiscale systems, especially  of the interaction of flow and elastic porous media. He and his collaborators made fundamental contributions to multiscale modeling of poro-elastic systems and their homogenization. He significantly contributed to mathematically rigorous derivation of Biot’s-systems \cite{ClopeauFerrinGilbertc2001,ferrin2003homogenizing,GilbertMikelic2000}, using multi-scale methods and linearized models for the viscid and inviscid flow and elasticity, and strongly promoted their application in a broad field of applications. 
E.g., in \cite{JaegerMikelicNeussRaduAnalysis,JaegerMikelicNeussRaduHomogenization} fluid-structure interactions in cell tissues is coupled with diffusion, transport and reactions in the cells and the extra-cellular space. Passing to a scale limit, a quasi-static Biot system coupled with the upscaled reactive  flow is obtained. Effective Biot’s coefficients depend on the reactant concentration. Furthermore, in \cite{JaegerMikelic1998}, effective laws for flows through a filter of finite thickness with rigid structure were derived, including a Darcy-type law for the flow through the filter, using the analysis of boundary layers. Andro Mikelic and his collaborators also brought essential contributions to the derivation of transmission laws at interface coupling different regimes. The necessary interface laws so far are rather often justified with phenomenological arguments. Mikelic demands in \cite{MikelicWheeler2012} their derivation with mathematical rigour, as far as possible: 
\textit{"The physical interpretation to be ascribed to these ad hoc interface and boundary conditions seems obscure. There is a need of obtaining interface and boundary conditions from first principles"}. An important example is the law of Beavers-Joseph \cite{BeaversJoseph1967} which was derived rigorously in \cite{JaegerMikelic1996}, and analyzed by numerical simulations in \cite{jager2001asymptotic}. Furthermore, in  \cite{marciniak2015rigorous} the quasi-static Biot's equations in a thin poro-elastic plate with prescribed boundary conditions was considered and  the dimension reduction  as the thickness tends to zero was investigated. 
\\
In this paper, effective equations for the interaction of a fluid with a thin porous elastic layer with thickness of order $\epsilon$ and a pore structure periodic in horizontal direction also of period $\epsilon$ are rigorously derived by passing to the two-scale limit for $\epsilon \to 0$. The fluid flow in the bulk regions and in the pores of the elastic layer is described by an instationary Stokes equation, whereas for the displacement of the solid part of the layer the system of linear elasticity is used. At the fluid-solid interface a linearized kinetic condition is assumed. This linearization is common to all existig results concerning the homogenization of fluid-structure interactions so far. The main contribution of this paper is the rigorous homogenization of the porous structure in the layer and the reduction of the layer to an interface $\Sigma$ in the two-scale limit $\epsilon \to 0$.  For the derivation of the macroscopic model we use the method of two-scale convergence for thin heterogeneous layers \cite{NeussJaeger_EffectiveTransmission}, which was introduced for homogeneous thin structures in \cite{MarusicMarusicPalokaTwoScaleConvergenceThinDomains}. However, for the treatment of problems in continuum mechanics involving thin porous layers new multiscale tools are necessary.
These are  formulated and derived in the form required here in \cite{GahnJaegerTwoScaleTools}, including extensions, Korn inequalities, two-scale compactness of $\epsilon$-dependent sets in Sobolev spaces, and analysis of two-scale limits.
The effective model consists of the Stokes equation  coupled to a time dependent plate equation on the interface $\Sigma$ including homogenized elasticity coefficients carrying information about the micro structure of the layer. In the zeroth order approximation we obtain continuity of the velocities at the interface. More precisely, only a vertical movement occurs while the tangential components vanish. The tangential movement in the solid is of order $\epsilon$ and given by a Kirchhoff-Love displacement.
To obtain some information about the fluid pressure and the first order approximation of the fluid velocity in the layer, in a second  step we derive higher order correctors for the fluid in the thin layer. In these orders of approximation the fluid velocity in the membrane is equal to the velocity of the solid. Hence, our results are an important first step that should be followed by the determination of the next term in an $\epsilon$-expansion, capturing also tangential and transversal fluxes relative to the movement of the solid phase in the thin porous layer. Determining this term of order $\epsilon^2$ is of particular importance to quantify the mass transport across the layer, and is part of our ongoing work.

Let us now indicate further literature contribution related to this work. For interactions of fluids with elastic structures, existence theorems without the restriction to linearized kinetic transmission conditions are available e.g., in \cite{Boulakia2007,CoutandShkoller2005,CoutandShkoller2006}, see also \cite{RaymondVanninathan2014} for more references, however, under assumptions that are not fulfilled in the problem at hand (like e.g., no-slip or periodic boundary conditions for the fluid). In \cite{muha2013existence} a fluid-structure problem for cylindrical flow described by the Navier-Stokes equations with a moving boundary given by a Koiter shell model is analyzed. However, the coupling condition between the fluid and the solid surface is based on phenomenological considerations. Our contribution is an essential step for  the rigorous derivation of such coupling conditions.
There is a large literature on effective laws for flows through inelastic sieves and filters, here we only mention some pioneering works. A stationary Stokes flow through an $\epsilon$-periodic filter consisting of an array of (disconnected) obstacles of size $\epsilon$ is treated in \cite{sanchez1982boundary} and \cite{ConcaI1987,ConcaII1987}. A similar geomertry is considered in \cite{BourgeatGipoulouxMarusicPaloka2001} for non-Newtonian flow. The case of tiny holes of 
order $\epsilon^2$ (for $n=3$) is treated in \cite{AllaireII1991} and $\epsilon^{\alpha}$ with $\alpha \in (1,2)$ in \cite{Marusic1998}.Dimension reduction for thin homogeneous elastic layers is quite standard, see for example \cite{ciarlet1997mathematical}. First results combining homogenization and dimension reduction with oscillating elasticity tensors have been established in \cite{caillerie1984thin}. However, results for perforated thin elastic structures seem to be rare. Here we have to mention the paper \cite{griso2020homogenization} which deals with the unfolding method for thin perforated structures in linear elasticity and gives a Korn-inequality for a special boundary condition slightly different from the situation considered in our setting.   In \cite{OrlikPanasenkoStavre2021} a dimension reduction for a thin (homogeneous) elastic stiff  plate separating two fluid bulk domains is performed. The scaling of the elasticity tensor is different from our setting and there is no fluid within the plate. However, rigorous results treating fluid flow through thin  porous elastic layers seem to be missing in the literature, and our paper is a significant contribution to close this gap.

Next, we give a short survey on the content of this paper:
The $\epsilon$-dependent microscopic model is formulated and discussed in Section \ref{SectionMicroMacroModelMainResult}. In Section \ref{SectionMainResult} we formulate the macroscopic model and the main result of the paper, see Theorem \ref{MainResult}, which includes the convergence results for the solutions of the microscopic model to the macroscopic solution.
Existence and uniqueness together with \textit{a priori} estimates for the solutions of the microscopic problems are derived in Section \ref{SectionApriori}. 
In Section \ref{SectionConvergenceMicroSolution} we prove the convergence results for the micro solutions, and based on  these results we derive the macroscopic problem including the cell problems in Section \ref{SectionDerivationMacroModel}. Higher order correctors are derived in Section \ref{SectionCorrectors}.
A conclusion in Section \ref{SectionConclusion} summarizes and discusses the achieved progress and open problems.
The Korn-inequality for perforated thin layers and an extension operator which in particular preserves the uniform \textit{a priori} bound for the symmetric gradient are given in Appendix \ref{SectionAppendixAuxiliary}.  Definitions and basic results related to the two-scale convergence are summarized in Appendix \ref{AppendixTwoScaleConvergence}.

\section{Microscopic model}
\label{SectionMicroMacroModelMainResult}

In this section we introduce and analyze the microscopic model. In a first step we introduce the necessary notations for the definition of the microscopic domain with the thin perforated membrane depending on the parameter $\epsilon$. On this microscopic domain we formulate the microscopic problem and introduce the weak formulation. We prove existence and uniqueness for the micro-model and show \textit{a priori} estimates uniformly with respect to $\epsilon$. These estimates are the basis for the derivation of the macroscopic model for $\epsilon \to 0$.

\subsection{The microscopic geometry}

We consider the domain $\oe := \Sigma \times (-H - \epsilon,H + \epsilon)$ with $H \gr 0$, and 
 $\Sigma = (a,b)\subset \R^2$ with $a,b \in \Z^2$ and $a_i \kl b_i$ for $i=1,2$. The domain $\oe$  consists of two bulk domains
\begin{align*}
\oe^+ := \Sigma \times (\epsilon,H + \epsilon), \mbox{ and } \quad \oe^- := \Sigma \times (-H - \epsilon,-\epsilon),
\end{align*}
which are separated by the thin layer
\begin{align*}
\oem := \Sigma \times (-\epsilon,\epsilon).
\end{align*}
Within the  thin layer we have a fluid part $\oemf$ and a solid part $\oems$, which have a periodical microscopic structure. More precisely, we define  the reference cell
\begin{align*}
Z := Y\times (-1,1) := (0,1)^2 \times (-1,1),
\end{align*}
with top and bottom
\begin{align*}
S^{\pm} := Y \times \{\pm 1\}.
\end{align*}
The cell $Z$ consists of a solid part $Z^s\subset Z $, see Figure \ref{FigureMicroDomain}, and a fluid part $Z^f \subset Z$ with common interface $\Gamma = \mathrm{int}\left( \overline{Z^s} \cap \overline{Z^f} \right)$. Hence, we have 
\begin{align*}
Z = Z^f \cup Z^s \cup \Gamma.
\end{align*}
We   assume that $S^{\pm} \cap \partial Z^s = \emptyset$. Furthermore, we request that $Z^f$ and $Z^s$ are open, connected   with Lipschitz-boundary, and the lateral boundary is $Y$-periodic which means that for $i=1,2$ and $\ast \in \{s,f\}$
\begin{align*}
\left(\partial Z^{\ast} \cap \{y_i = 0\}\right) + e_i = \partial Z^{\ast} \cap \{y_i=1\}.
\end{align*}
We introduce the set $K_{\epsilon}:= \{k \in \Z^2 \times \{0\} \, : \, \epsilon(Z + k) \subset \oem\}$. Clearly, we have $\oem = \mathrm{int}\left(\bigcup_{k\in K_{\epsilon}} \epsilon(\overline{Z} + k)\right)$.
Now, we define the fluid and solid part of the membrane, see Figure \ref{FigureMicroDomain}, by
\begin{align*}
\oemf &:= \mathrm{int} \left( \bigcup_{k \in K_{\epsilon}} \epsilon \left(\overline{Z^f} + k \right) \right),
\\
\oems &:= \mathrm{int} \left( \bigcup_{k \in K_{\epsilon}} \epsilon \left(\overline{Z^s} + k \right) \right).
\end{align*}
The fluid-structure interface between the solid and the fluid part is denoted by 
\begin{align*}
\geps := \mathrm{int}\left( \overline{\oems} \cap \overline{\oemf}\right).
\end{align*}
The interface between the fluid part in the membrane and the bulk domains is defined by
\begin{align*}
\sepm := \Sigma \times \{\pm \epsilon\}.
\end{align*}
Altogether, we have the following decomposition of the domain $\oe$
\begin{align*}
\oe &= \oe^+ \cup \oe^- \cup \oem \cup  S_{\epsilon}^+ \cup S_{\epsilon}^-
\\
&= \oe^+ \cup \oe^- \cup \oems \cup \oemf \cup \geps \cup  S_{\epsilon}^+ \cup S_{\epsilon}^-.
\end{align*}
The whole fluid part is defined by
\begin{align*}
\oef := \oe \setminus \overline{\oems}.
\end{align*}

\begin{figure}
\begin{minipage}{3.3cm}
\centering
\includegraphics[scale=0.425]{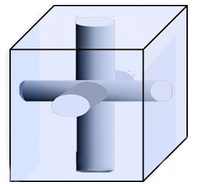}
\end{minipage} 
$\boldsymbol{\longrightarrow} $ 
\begin{minipage}{4.02cm}
\centering
\includegraphics[scale=0.2525]{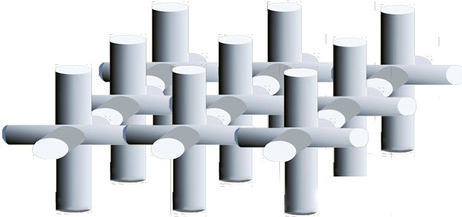}
\end{minipage}
\hspace{.1cm} $\boldsymbol{\longrightarrow}$ 
\begin{minipage}{4.7cm}
\centering
\includegraphics[scale=0.225]{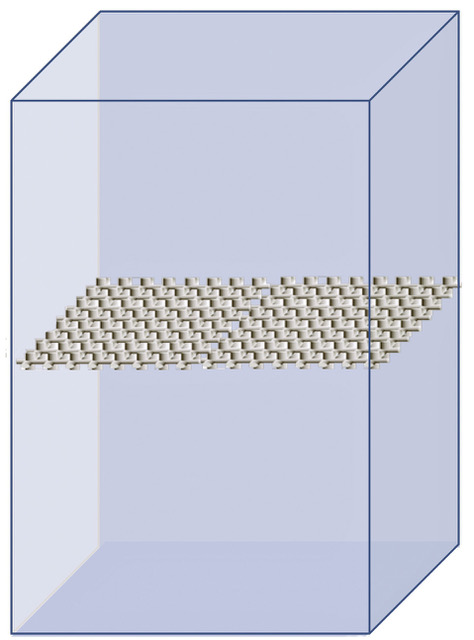}
\end{minipage}
\caption{Left: A reference cell $Z$ for the porous elastic layer with the solid part $Z^s$ highlighted by the dark coloring. Middle: A solid perforated layer generated by the periodically repeated standard solid part $Z^s$. Right: The microscopic domain $\oe$ with the perforated layer $\oem$ consisting of the fluid part $\oemf$ and the solid part $\oems$.}
 \label{FigureMicroDomain}
\end{figure}

By construction we have that $\oef$, $\oemf$, and $\oems$ are connected. Further we assume that these domains are Lipschitz.
Now, we split the boundary $\partial \oe$ in several parts ($\ast \in \{f,s\}$): 
\begin{align*}
\partial_N \oe &:= \Sigma \times \{\pm (H+\epsilon)\},
\\
\partial_D \oe^+ &:= \partial \Sigma \times (\epsilon,H + \epsilon),
\\
\partial_D \oe^- &:= \partial \Sigma \times (-H - \epsilon,-\epsilon),
\\
\partial_D \oe &:= \partial \Sigma \times (-H - \epsilon, H + \epsilon),
\\
\partial_D \Omega^+ &:= \partial \Sigma \times (0,H),
\\
\partial_D \Omega^- &:= \partial \Sigma \times (-H,0),
\\
\partial_D \oem &:= \partial \Sigma \times (-\epsilon,\epsilon).
\\
\partial_D \oe^{M,\ast} &:= \mathrm{int} \left(\partial_D \oem \cap \partial \oe^{M,\ast}\right),
\\
\partial_D \oef &:= \partial_D \oemf \cup \bigcup_{\pm} \partial_D \oepm.
\end{align*}
In the limit $\epsilon \to 0 $ the thin layer $\oem$ is reduced to the interface $\Sigma$ and the domains $\oe$  resp. $\oepm$ converge to the  macro domains  $\Omega$  resp. $\Omega^{\pm}$ defined by
\begin{align*}
\Omega &:= \Sigma \times (-H,H),
\\
\Omega^+ &:= \Sigma \times (0,H),
\\
\Omega^- &:= \Sigma \times (-H,0).
\end{align*}
The Dirichlet- and Neumann-part of the macroscopic boundary $\partial \Omega$ is denoted by
\begin{align*}
\partial_D \Omega &:= \partial \Sigma \times (-H,H),
\\
\partial_N \Omega &:= \Sigma \times \{\pm H\}.
\end{align*}

\textbf{Notations: } For an arbitrary function $\phieps : \oef \rightarrow \R^m$ for $m\in \N$ we define the restrictions to the bulk domains and the fluid part of the membrane by
\begin{align*}
\phieps^{\pm} := \phieps\vert_{\oe^{\pm}}, \qquad \phieps^M := \phieps\vert_{\oemf}.
\end{align*}
Function spaces with the index $\#$ denote spaces which are $Y$-periodic. Especially
we define the space of smooth and $Y$-periodic functions by
\begin{align*}
C_{\#}^{\infty}(\overline{Z}):= \left\{ v \in C^{\infty}(\R^2 \times [-1,1]) \, : \, v(\cdot + e_i ) = v \mbox{ for } i=1,2\right\},
\end{align*}
and  $H^1_{\#}(Z)$ is the closure of $C_{\#}^{\infty}(\overline{Z})$ with respect to the usual $H^1$-norm. The space $H^1_{\#}(Z^s)$ is defined by restriction of functions from $H^1_{\#}(Z)$.

\subsection{The microscopic problem}
In the fluid part $\oef$ we have the fluid velocity $\veps = (\veps^+,\veps^M,\veps^-): (0,T)\times \oef \rightarrow \R^3$ and the fluid pressure $\peps = (\peps^+,\peps^M,\peps^-) : (0,T)\times  \oef \rightarrow \R$. The displacement of the solid part is given by $\ueps : (0,T)\times \oems \rightarrow \R^3$. We consider the following fluid-structure interaction problem on $\oe$:

The evolution of the  velocity and pressure of the fluid is given by
\begin{subequations}\label{MicroscopicModel}
\begin{align}
\partial_t \veps^{\pm}  - \nabla \cdot D( \veps^{\pm}) + \nabla \peps^{\pm}  &= f_{\epsilon}^{\pm} &\mbox{ in }& (0,T)\times \oepm,
\\
\foe\partial_t \veps^M - \foe \nabla \cdot D( \veps^M )+  \foe\nabla \peps^M  &= \foe f_{\epsilon}^M &\mbox{ in }& (0,T)\times \oemf,
\\
\nabla \cdot \veps &= 0 &\mbox{ in }& (0,T)\times \oef,
\\
\left(-\peps I + D(\veps) \right)\cdot \nu &= 0 &\mbox{ on }& (0,T) \times \partial_N \oe,
\\
\label{MicroModelBCDirichlet}
\veps &= 0 &\mbox{ on }& (0,T) \times \partial_D \oef,
\\
\veps(0) &= \veps^0 &\mbox{ in }& \oef,
\end{align}
with the symmetric gradient $D(\ueps):= \frac12 \left(\nabla \ueps + \nabla \ueps^T\right)$.
On the fluid-fluid-interface between the bulk domains $\oe^{\pm}$ and the fluid part of the membrane $\oemf$ we assume continuity of the fluid-velocity and the normal stresses
\begin{align}
\veps^{\pm} &= \veps^M &\mbox{ on }& (0,T)\times \sepm,
\\
(-\peps^{\pm} I +  D( \veps^{\pm}) ) \cdot \nu &=\foe (-\peps^M + D( \veps^M) ) \cdot \nu   &\mbox{ on }& (0,T)\times \sepm.
\end{align} 
The displacement $\ueps : (0,T)\times \oems \rightarrow \R^n$ is described by 
\begin{align}
\foe\partial_{tt} \ueps - \dfrac{1}{\epsilon^3}\nabla \cdot \big(A_{\epsilon} D( \ueps )\big) &= 0 &\mbox{ in }& (0,T)\times \oems,
\\
\ueps &= 0 &\mbox{ on }& (0,T)\times \partial_D \oems,
\\
\ueps(0) = \partial_t \ueps(0) &= 0 &\mbox{ in }& \oems.
\end{align}
On the microscopic interface $\geps$ between the fluid and solid we assume the following linearized conditions
\begin{align}
\veps^M &= \partial_t \ueps &\mbox{ on }& (0,T)\times \geps,
\\
\foe\left(-\peps^M I + D( \veps^M) \right)\cdot \nu &= \dfrac{1}{\epsilon^3} A_{\epsilon} D(\ueps) \cdot \nu &\mbox{ on }& (0,T)\times \geps.
\end{align}
\end{subequations}
In many applications it might be necessary to consider an inhomogeneous inflow  boundary condition. In the following remark we identify a class of boundary conditions, which are covered by our model.
\begin{remark}
Our model also includes the case of some specific inhomogeneous boundary conditions on $\partial_D \oepm$. In fact, if we consider in $\eqref{MicroModelBCDirichlet}$ the condition
\begin{align*}
\veps^{\pm} = v_{D,\epsilon}^{\pm} \qquad\mbox{on } (0,T)\times \partial_D \oepm,
\end{align*}
with $v_{D,\epsilon}^{\pm} = v_D^{\pm} (\cdot \mp e_3)$  and $v_D^{\pm}$ defined on $\partial_D \Omega^{\pm}$, this inhomogeneous problem can be transformed to our model $\eqref{MicroscopicModel}$ with no-slip condition on $\partial_D \oepm$, if there exists an extension of $v_D^{\pm}$ to the bulk domain $\Omega^{\pm}$, such that $v_D^{\pm} \in H^1((0,T),H^2(\Omega^{\pm}))^3 \cap H^2((0,T),L^2(\Omega^{\pm}))^3$ and 
\begin{align*}
\nabla \cdot v_D^{\pm} &= 0 &\mbox{ in }& (0,T)\times \Omega^{\pm},
\\
v_D^{\pm}  &= 0 &\mbox{ on }& (0,T)\times \sepm,
\end{align*}
and the initial condition fulfills also $\veps^0 = v_{D,\epsilon}^{\pm}$ on $\partial_D \oepm$, see also the assumption \ref{AssumptionInitialConditionsFluid}.
Such an extension exists for example  (for $\epsilon$ small enough) if $v_D^{\pm} \in H^2((0,T),H^{\frac32}(\partial_D \Omega^{\pm}))^3$ with compact support on each side of $\partial_D \Omega^{\pm}$. In fact, by using arguments as in \cite[Proof of Theorem 5.4]{fabricius2017stokes} we can extend $v_D^{\pm}$ to the whole boundary $\partial \Omega^{\pm}$ such that $v_D^{\pm} = 0$ on $\sepm$ and 
\begin{align}
    \int_{\partial \Omega^{\pm}} v_D^{\pm} \cdot \nu  = 0.
\end{align}
Smoothing the edges and nodes of $\Omega^{\pm}$, due to the compact support of $v_D^{\pm}$, we can consider $\Omega^{\pm}$ as a smooth domain. From \cite[Corollario 1]{RSMUP_1961__31__308_0}, see also \cite[Chapter III, Theorem 1.5.1]{Sohr}, we obtain the existence of a divergence free extension $v_D^{\pm}$ to the whole domain $\Omega^{\pm}$ with $v_D^{\pm} \in H^2((0,T),H^2(\Omega^{\pm})^3$.
\end{remark}

The weak formulation of the microscopic model $\eqref{MicroscopicModel}$ reads as follows:  We say that the triple $(\veps,\peps,\ueps)$ is a weak solution of the microscopic model $\eqref{MicroscopicModel}$, iff
\begin{align*}
\veps &\in L^2((0,T),H^1(\oef)^3)\cap H^1((0,T),L^2(\oef)^3),
\\
\peps &\in L^2((0,T),L^2(\oef)),
\\
\ueps &\in H^1((0,T),H^1(\oems)^3)\cap H^2((0,T),L^2(\oems)^3),
\end{align*} 
with $\veps = 0$ on $\partial_D \oe$, $\ueps = 0$ and $\partial_t \ueps = 0$ on $\partial_D \oems$, and  $\veps^M = \partial_t \ueps $ on $\geps$ and 
\begin{align}
\begin{aligned}
\label{EQ:Var_Micro}
\sum_{\pm}\int_{\oepm}& \partial_t \veps^{\pm} \phi dx + \foe \int_{\oemf} \partial_t \veps^M \phi dx  + \foe \int_{\oems} \partial_{tt} \ueps \phi dx 
\\
&+ \sum_{\pm} \int_{\oepm}  D( \veps^{\pm}) : D( \phi) dx  + \foe \int_{\oemf} D( \veps^M) : D( \phi) dx + \frac{1}{\epsilon^3} \int_{\oems} A_{\epsilon} D(\ueps ) : D(\phi) dx  
\\
&- \sum_{\pm} \int_{\oepm} \peps^{\pm} \nabla \cdot \phi dx - \foe \int_{\oemf} \peps^M \nabla \cdot \phi dx 
\\
&=\sum_{\pm} \int_{\oepm} \feps^{\pm} \phi dx + \foe \int_{\oemf} \feps^M \phi dx ,
\end{aligned}
\end{align}
for all $\phi \in H^1(\oe)^3$ with $\phi = 0 $ on $\partial_D \oe$.
\\

\noindent\textbf{Assumptions on the data: }
\begin{enumerate}
[label = (A\arabic*)]
\item\label{AssumptionElasticityTensor} The elasticity tensor $A_{\epsilon}$ is defined by $A_{\epsilon}(x) := A\left(\fxe\right)$ with $A\in L_{\#}^{\infty}(Z^s)^{3\times 3 \times 3 \times 3}$ symmetric and coercive on the space of symmetric matrices, more precisely for $i,j,k,l=1,2,3$
\begin{align*}
A_{ijkl} &= A_{jilk} = A_{ljik},
\\
A(y) B : B &\geq c_0 \vert B\vert^2 \quad \mbox{ for almost  every } y \in Z,
\end{align*}
with $c_0 \gr 0$ and all $B \in \R^{3\times 3}$ symmetric.
\item  There exists $f^{\pm} \in  H^1((0,T),L^2(\Omega^{\pm}))^3$, such that $f_{\epsilon}^{\pm} = f^{\pm}(\cdot \mp \epsilon e_3 )$.
\item  It holds that $f_{\epsilon}^M \in  H^1((0,T),L^2(\oemf))^3$ with
\begin{align*}
\frac{1}{\sqrt{\epsilon}} \Vert f_{\epsilon}^M \Vert_{L^2((0,T)\times \oemf)} + \frac{1}{\sqrt{\epsilon}} \Vert \partial_t f_{\epsilon}^M \Vert_{L^2((0,T)\times \oemf)} \le C.
\end{align*}
Further, there exists $f_0^M \in L^2((0,T) \times \Sigma \times Z)^3$ such that
\begin{align*}
\chi_{\oemf}f_{\epsilon}^M \rats \chi_{Z^f}f_0^M.
\end{align*}
\item\label{AssumptionInitialConditionsFluid} The initial condition $\veps^0$ fulfills
\begin{align*}
\veps^0 = \begin{cases}
\veps^{0,\pm} &\mbox{ in } \oepm,
\\
\veps^{0,M} &\mbox{ in } \oemf,
\end{cases}
\end{align*}
with $\veps^{0,\pm} \in H^1(\oepm)^3$ and  $v^{0,\pm} \in H^1(\Omega^{\pm})^3$, and $\veps^{0,M} \in H^1(\oemf)^3$
 such that 
$\veps^0$ fulfills the following compatibility condition: There exists $p_{\epsilon}^0:=(\peps^{0,+},\peps^{0,M},\peps^{0,-})$ with $p_{\epsilon}^{0,\pm} \in L^2((0,T)\times \oepm)$ and $p_{\epsilon}^{0,M} \in L^2((0,T)\times \oemf)$ such that $(\veps^0,\peps^0)$ is the weak solution of 
\begin{align*}
-\nabla \cdot D(\veps^{0,\pm}) + \nabla \peps^{0,\pm} &= F_{\epsilon}^{0,\pm} &\mbox{ in }& \oepm,
\\
-\foe \nabla \cdot D(\veps^{0,M}) + \foe \nabla \peps^{0,M} &= \foe F_{\epsilon}^{0,M} &\mbox{ in }& \oemf,
\\
\nabla \cdot \veps^0 &= 0 &\mbox{ in }& \oef,
\\
\veps^0 &= 0 &\mbox{ on }& \partial_D \oef ,
\\
\big(-p_{\epsilon}^0 + D(\veps^0)\big)\cdot \nu &= 0 &\mbox{ on }& \partial_N \oef \cup \geps,
\\
\veps^{0,\pm} &= \veps^{0,M} &\mbox{ on }& (0,T)\times \sepm,
\\
(-\peps^{0,\pm} I +  D( \veps^{0,\pm}) ) \cdot \nu &=\foe (-\peps^{0,M} + D( \veps^{0,M}) ) \cdot \nu   &\mbox{ on }& (0,T)\times \sepm,
\end{align*}
with $F_{\epsilon}^{0,\pm} \in L^2(\oepm)^3$ and $F_{\epsilon}^{0,M} \in L^2(\oemf)^3$ such that
\begin{align*}
\Vert F_{\epsilon}^{0,\pm} \Vert_{L^2(\oepm)} + \frac{1}{\sqrt{\epsilon}} \Vert F_{\epsilon}^{0,M} \Vert_{L^2(\oemf)} \le C.
\end{align*}
By standard energy estimates (similar to the proofs of Lemma \ref{AprioriEstimatesLemma}) and the Korn-inequality for functions vanishing on $\geps$, see also \cite[Chapter 4, Theorem 4.5]{Oleinik1992} , we get 
\begin{align*}
    \Vert \veps^{0,\pm} \Vert_{H^1(\oepm)} + \frac{1}{\sqrt{\epsilon}} \Vert \veps^{0,M} \Vert_{H^1(\oemf)} \le C,
\end{align*}
and we assume there exists $v^{0,\pm} \in H^1(\Omega^{\pm})^3$ with $v^{0,\pm} = 0$ on $\partial_D \Omega^{\pm}$, such that (for the whole sequence)
\begin{align*}
    \veps^{0,\pm}(\cdot \pm \epsilon e_3) &\rightharpoonup v^{0,\pm} \qquad \mbox{ weakly in } H^1(\Omega^{\pm})^3,
    \\
    \veps^{0,M} &\rats 0.
\end{align*}
We emphasize that the two-scale convergence of $\veps^{0,M}$ to zero is a direct consequence of the no-slip condition on $\geps$.
\end{enumerate}

The aim of this paper is the derivation of a macroscopic model with  effective interface conditions for $\epsilon \to 0$, when the thin layer reduces to the interface $\Sigma$. The principal idea is to assume that the microscopic solution fulfills a two-scale ansatz. We illustrate this ansatz for the displacement in the layer:
\begin{align}\label{TwoScaleAnsatz}
\ueps(t,x) = u_0\tbxfxe + \epsilon u_1 \tbxfxe + \epsilon^2 u_2\tbxfxe + \ldots,
\end{align}
with functions $u_j$ which are $Y$-periodic with respect to the variable $y=\fxe$. The two-scale convergence gives a rigorous justification of the expansion  in $\eqref{TwoScaleAnsatz}$. We will identify the expansion for the displacement up to order 2, whereas for the fluid velocity we get the terms up to order 1.

\section{Statement of the main results}
\label{SectionMainResult}
The aim of the paper is the derivation of a macroscopic model on $\Omega$ for $\epsilon \to 0$ when the thin layer $\oem$ is reduced to the interface $\Sigma$. 
We show that the microscopic solutions $(\veps,\peps,\ueps)$ convergence in a suitable sense to the solution of the macroscopic model. The crucial point is the derivation of interface laws on $\Sigma$, which consist of a time dependent plate equation with effective elasticity coefficients arising due to homogenization effects, and effective coupling conditions between the velocities of the two phases.

\subsection{The macroscopic model}
We start with the formulation of the macroscopic model. Hereby, we use the notation: 
\begin{align*}
\Delta_{\x} : \big(b^{\ast} D_{\x}(\tilde{u}_1) +c^{\ast} \nabla_{\x}^2 u_0^3\big) := \sum_{i,j,k,l=1}^2 \partial_{kl}\left(b^{\ast}_{ijkl} D_{\x}(\tilde{u}_1)_{ij} +  c^{\ast}_{ijkl}  \partial_{ij} u_0^3\right).
\end{align*}
Then, the macroscopic model in the strong formulation reads as follows:  
Find $v_0^{\pm} : (0,T)\times \Omega^{\pm} \rightarrow \R^3$, $p_0^{\pm}:(0,T)\times \Omega^{\pm} \rightarrow \R$, $u_0^3:(0,T)\times \Sigma \rightarrow \R$, and $\tilde{u}_1: (0,T)\times \Sigma \rightarrow \R^2$, such that
\begin{align}
\begin{aligned}\label{MacroModelStrongFormulation}
\partial_t v_0^{\pm} - \nabla \cdot \left(D(v_0^{\pm})\right) + \nabla p_0^{\pm} &= f_0^{\pm} &\mbox{ in }& (0,T)\times \Omega^{\pm},
\\
\nabla \cdot v_0^{\pm} &= 0 &\mbox{ in }& (0,T)\times \Omega^{\pm},
\\
v_0^{\pm} &= 0 &\mbox{ on }& (0,T)\times \partial_D \Omega,
\\
\left(-D(v_0^{\pm}) + p_0^{\pm} I \right) \nu &= 0 &\mbox{ on }& (0,T)\times \partial_N \Omega,
\\
v_0^{\pm} &= (0,0,\partial_t u_0^3)^T &\mbox{ on }& (0,T)\times \Sigma,
\\
-\nabla_{\x} \cdot \left(a^{\ast} D_{\x}(\tilde{u}_1) + b^{\ast} \nabla_{\x}^2 u_0^3 \right) &= 0 &\mbox{ in }& (0,T)\times \Sigma,
\\
\partial_{tt} u_0^3 + \Delta_{\x} : \left(b^{\ast} D_{\x}(\tilde{u}_1) + c^{\ast}\nabla_{\x}^2 u_0^3 \right) &= \int_{Z^f} f_0^{3,M} dy +  \left(\left\lsem  - D(v_0^{\pm}) + p_0^{\pm}I  \right\rsem\nu\right)_3 &\mbox{ in }& (0,T)\times \Sigma,
\\
u_0^3 = \nabla_{\x} u_0^3 \cdot \nu & = 0&\mbox{ on }& (0,T) \times \partial \Sigma,
\\
\tilde{u}_1 &= 0 &\mbox{ on }& (0,T)\times \partial \Sigma,
\end{aligned}
\end{align}
where $\left\lsem  - D(v_0^{\pm}) + p_0^{\pm}I  \right\rsem\nu$ denotes the jump of the stresses across $\Sigma$, and $a^{\ast},\, b^{\ast},\, c^{\ast} \in \R^{2\times 2 \times 2 \times 2}$ are the homogenized elasticity tensors defined in $\eqref{HomogenizedTensors}$ via solutions of cell problems (see $\eqref{CellProblemStandard}$ and $\eqref{CellProblemHesse}$). Further, we have the the initial conditions
\begin{align}
\begin{aligned}
\label{InitialConditionsMacroModel}
u_0^3(0) &= 0 &\mbox{ in }& \Sigma,
\\
 \partial_t u_0^3(0) &= 0  &\mbox{ in }& \Sigma,
\\
v_0^{\pm}(0) &= v^{0,\pm} &\mbox{ in }& \Omega^{\pm}.
\end{aligned}
\end{align}
Let us now give the weak formulation for this problem. We define the space
\begin{align*}
\spaceH:= \left\{ \phi \in H^1(\Omega)^3\, : \, \phi = 0 \mbox{ on } \partial_D \Omega, \, \phi\vert_{\Sigma} = (0,0,\phi^3)\vert_{\Sigma}^T \in H^2_0(\Sigma)^3\right\}.
\end{align*}
We say that $(v_0^{\pm},p_0^{\pm},\tilde{u}_1,u_0^3)$ is a weak solution of the problem $\eqref{MacroModelStrongFormulation}$ if
\begin{align*}
v_0^{\pm} &\in L^2((0,T),H^1(\Omega^{\pm}))^3 \cap H^1((0,T),L^2(\Omega^{\pm}))^3,
\\
p_0^{\pm} &\in L^2((0,T)\times \Omega^{\pm})
\\
u_0^3 &\in H^1((0,T),H^2_0(\Sigma))\cap H^2((0,T),L^2(\Sigma)),
\\
\tilde{u}_1 &\in H^1((0,T),H_0^1(\Sigma))^2,
\end{align*}
with $(v_0^+,v_0^-) \in L^2((0,T),\spaceH) $, $\nabla \cdot v_0^{\pm}$, and for all $V\in \spaceH$ and $\bar{U} \in H_0^1(\Sigma)^2$ it holds almost everywhere in $(0,T)$
\begin{align}
\begin{aligned}
\label{VarFormMacroModel}
\sum_{\pm} \int_{\Omega^{\pm}} &\partial_t v_0^{\pm} \cdot V    dx  + \int_{\Sigma} \partial_{tt} u_0^3 V^3   d\x  
+ \sum_{\pm}\left[ \int_{\Omega^{\pm}} D(v_0^{\pm}) : D(V) dx  
-   \int_{\Omega^{\pm}} p_0^{\pm} \nabla \cdot V  dx  \right]
\\
+&    \int_{\Sigma} a^{\ast} D_{\x}(\tilde{u}_1) : D_{\x}(\bar{U}) + b^{\ast} \nabla_{\x}^2 u_0^3 : D_{\x}(\bar{U}) + b^{\ast} D_{\x}(\tilde{u}_1) : \nabla_{\x}^2 V^3 + c^{\ast} \nabla_{\x}^2 u_0^3 : \nabla_{\x}^2 V^3 d\x 
\\
=&  \sum_{\pm} \int_{\Omega^{\pm}} f_0^{\pm} \cdot V  dx   +   \int_{\Sigma} \int_{Z^f} f_0^{3,M} V^3  dy d\x ,
\end{aligned}
\end{align}
together with the initial conditions in $\eqref{InitialConditionsMacroModel}$.

\subsection{Main theorem}
Now we are able to formulate the main theorem of our paper. For the definition of the two-scale convergence see Appendix \ref{AppendixTwoScaleConvergence}.
\begin{theorem}\label{MainResult}
For the microscopic solution $(\veps,\peps,\ueps)$ the following convergence result hold. In the bulk domains we have that
\begin{align*}
\veps^{\pm}(\cdot_t , \cdot_x \pm \epsilon e_3) &\rightharpoonup v_0^{\pm}  &\mbox{ weakly in }& L^2((0,T),H^1(\Omega^{\pm}))^3,
\\
\partial_t \veps^{\pm}(\cdot_t , \cdot_x \pm \epsilon e_3) &\rightharpoonup \partial_t v_0^{\pm}  &\mbox{ weakly in }& L^2((0,T),L^2(\Omega^{\pm}))^3,
\\
\peps^{\pm}(\cdot_t , \cdot_x \pm \epsilon e_3) &\rightharpoonup p_0^{\pm} &\mbox{ weakly in }& L^2((0,T),L^2(\Omega^{\pm})),
\end{align*}
whereas in the thin layer it holds for $\alpha = 1,2$, that 
\begin{align*}
\chi_{\oems} \frac{\ueps^{\alpha}}{\epsilon} &\rats \chi_{Z^s} (\tilde{u}_1^{\alpha} - y_3 \partial_{\alpha} u_0^3),
\\
\chi_{\oems} \ueps^3 &\rats \chi_{Z^s} u_0^3,
\\
\chi_{\oems} \partial_{tt} \ueps^3 &\rats \chi_{Z^s} \partial_{tt} u_0^3,
\\
\foe \chi_{\oems} D(\ueps) &\rats \chi_{Z^s} \big(D_{\x}(\tilde{u}_1) - y_3 \nabla_{\x}^2 u_0^3 + D_y(u_2)\big),
\\
\chi_{\oemf}\veps^M &\rats \chi_{Z^f} (0,0,\partial_t u_0^3)^T,
\\
\chi_{\oemf}\partial_t \veps^M &\rats \chi_{Z^f}(0,0,\partial_{tt} u_0^3)^T,
\end{align*}
where $u_2$ is a corrector term defined in Proposition \ref{PropositionCorrectorU2Representation} and $(v_0^{\pm},p_0^{\pm},\tilde{u}_1,u_0^3)$ is the unique weak solution of the macroscopic model $\eqref{MacroModelStrongFormulation}$.
%
\end{theorem}
The proof of the convergence results can be found in Section \ref{SectionConvergenceMicroSolution} and the limit model is derived in Section \ref{SectionDerivationMacroModel}.

\begin{remark}
To keep the setting simpler we assumed $S^{\pm}\cap Z^s = \emptyset$. However, Theorem \ref{MainResult} remains valid if $S^{\pm}\cap Z^s \neq \emptyset$. For this we need additional coupling conditions for the solid and the bulk fluid, where we consider again  continuity of the velocity and the stress. The main difference in the proof of Theorem \ref{MainResult} is the derivation of the cell problems $\eqref{CellProblemStandard}$ and $\eqref{CellProblemHesse}$, where we have to choose other types of test functions, see Remark \ref{RemarkZellprobleme}. 
\end{remark}

\section{Existence of the microscopic solution and \textit{a priori} estimates}
\label{SectionApriori}

To pass to the limit $\epsilon \to 0$ in the microscopic problem, we need uniform estimates with respect to $\epsilon$, which are obtained by standard energy estimates. However, the crucial point is to figure out the precise dependence on $\epsilon$. First of all, let us formulate an existence and uniqueness result.
\begin{proposition}
There exists a unique weak solution of the microscopic problem $\eqref{MicroscopicModel}$.
\end{proposition}
\begin{proof}
Existence is obtained by using a standard Galerkin approximation using similar \textit{a priori}  estimates as in Lemma \ref{AprioriEstimatesLemma} below. Unqiueness follows by standard energy estimates.
\end{proof}

\begin{lemma}\label{AprioriEstimatesLemma}
The microscopic solution $(\veps,\peps,\ueps)$ of problem $\eqref{MicroscopicModel}$ fulfills the following \textit{a priori} estimates:
\\
For the fluid velocity and pressure in the bulk domains $\oepm$ it holds that
\begin{align*}
\Vert \partial_t \veps^{\pm} \Vert_{L^{\infty}((0,T),L^2(\oepm))} + \Vert \veps^{\pm} \Vert_{L^{\infty}((0,T),L^2(\oepm))} + \Vert \nabla \veps^{\pm} \Vert_{L^{\infty}((0,T),L^2(\oepm))} &\le C,
\\
\Vert \peps^{\pm} \Vert_{L^{\infty}((0,T),L^2(\oepm))} &\le C.
\end{align*}
The fluid velocity and pressure in the fluid part of the layer $\oemf$ fulfills
\begin{align*}
\frac{1}{\sqrt{\epsilon}} \Vert \partial_t \veps^M \Vert_{L^{\infty}((0,T),L^2(\oemf))} + \frac{1}{\sqrt{\epsilon}} \Vert \veps^M \Vert_{L^{\infty}((0,T),L^2(\oemf))} \hspace{6em}&
\\
+ \frac{1}{\sqrt{\epsilon}} \Vert D(\veps^M)\Vert_{L^{\infty}((0,T),L^2(\oemf))} + \sqrt{\epsilon} \Vert \nabla \veps^M \Vert_{L^{\infty}((0,T),L^2(\oemf))} &\le C,
\\
\frac{1}{\sqrt{\epsilon}} \Vert \peps^M \Vert_{L^{\infty}((0,T),L^2(\oemf))} &\le C.
\end{align*}
For the displacement in the solid part of the layer $\oems$ it holds that
\begin{align*}
\frac{1}{\sqrt{\epsilon}}& \Vert \partial_{tt} \ueps \Vert_{L^{\infty}((0,T),L^2(\oems))}
\\
&+ \frac{1}{\sqrt{\epsilon}} \Vert \ueps^3 \Vert_{W^{1,\infty}((0,T),L^2(\oems))} 
+ \frac{1}{\epsilon^{\frac32} } \sum_{\alpha =1}^2 \Vert \ueps^{\alpha} \Vert_{W^{1,\infty}((0,T),L^2(\oems))} 
\\
& + \frac{1}{ \sqrt{\epsilon}} \Vert \nabla \ueps \Vert_{W^{1,\infty}((0,T),L^2(\oems))}  + \frac{1}{\epsilon^{\frac32}} \Vert D(\ueps) \Vert_{W^{1,\infty}((0,T),L^2(\oems))} \le C.
\end{align*}
\end{lemma}
\begin{proof} We separate the proof in several steps:

\noindent\textbf{Step 1:} As a test-function in $\eqref{EQ:Var_Micro}$ we use $\phi = \veps $ in $\oef$ and $\phi = \partial_t \ueps $ in $\oems$ to obtain  almost everywhere $  (0,T)$
\begin{align*}
\sum_{\pm}\frac{1}{2}\frac{d}{dt} &\Vert \veps^{\pm}\Vert_{L^2(\oepm)}^2 + \frac{1}{2\epsilon}\frac{d}{dt} \Vert \veps^M \Vert_{L^2(\oemf)}^2 +\frac{1}{2\epsilon}\frac{d}{dt} \Vert \partial_t \ueps \Vert_{L^2(\oems)}^2
\\
&+ \sum_{\pm} \Vert D(\veps^{\pm}) \Vert^2_{L^2(\oepm)} + \foe \Vert \veps^M \Vert_{L^2(\oemf)}^2 + \frac{1}{\epsilon^3} \frac12 \frac{d}{dt} \int_{\oems} A_{\epsilon} D(\ueps): D(\ueps) dx 
\\
=& \sum_{\pm}\int_{\oepm} f_{\epsilon}^{\pm} \cdot  \veps^{\pm} dx + \foe \int_{\oemf} f_{\epsilon}^M  \cdot \veps^M dx
\\
\le&  C \sum_{\pm}\Vert \veps^{\pm} \Vert_{L^2(\oepm)} + \frac{C}{\sqrt{\epsilon} } \Vert \veps^M \Vert_{L^2(\oemf)}.
\end{align*}
Integration with respect to time and using the coercivity  and continuity of $A_{\epsilon}$ from assumption \ref{AssumptionElasticityTensor}, we obtain for almost every $t \in (0,T)$ (with $\ueps(0) = \partial_t \ueps(0) = 0$)
\begin{align}
\begin{aligned}
\label{Step1Inequality}
\frac12 \sum_{\pm}\Vert \veps^{\pm}(t) &\Vert_{L^2(\oepm)}^2 + \frac{1}{2\epsilon} \Vert \veps^M (t) \Vert_{L^2(\oemf)}^2 + \frac{1}{2\epsilon} \Vert \partial_t \ueps(t) \Vert_{L^2(\oems)}^2
\\
&+ \sum_{\pm}\Vert D(\veps^{\pm})\Vert_{L^2((0,t)\times \oepm)}^2 + \foe \Vert D(\veps^M) \Vert_{L^2((0,t)\times \oemf)}^2 + \frac{c_0}{2\epsilon^3} \Vert D(\ueps)(t)\Vert_{L^2(\oems)}^2
\\
&\le  C \left( 1 +  \sum_{\pm}\Vert \veps^{\pm} \Vert_{L^2((0,t)\times \oepm)}^2 + \foe \Vert \veps^M \Vert_{L^2((0,t)\times \oemf)}^2 \right)
\\
&+ \frac12 \left( \sum_{\pm} \Vert \veps^{0,\pm} \Vert_{L^2(\oepm)}^2 + \foe \Vert \veps^{0,M} \Vert_{L^2(\oemf)}^2 \right)
\end{aligned}
\end{align}
Assumption \ref{AssumptionInitialConditionsFluid} and the Gronwall-inequality imply
\begin{align*}
\sum_{\pm}\Vert &\veps^{\pm} \Vert_{L^{\infty}((0,T),L^2(\oepm))} + \frac{1}{\sqrt{\epsilon}} \Vert \veps^M \Vert_{L^{\infty}((0,T),L^2(\oemf))} + \frac{1}{\sqrt{\epsilon}} \Vert \partial_t \ueps \Vert_{L^{\infty}((0,T),L^2(\oems))} 
\\
&+ \sum_{\pm}\Vert D(\veps^{\pm})\Vert_{L^2((0,T)\times \oepm)} + \frac{1}{\sqrt{\epsilon}} \Vert D(\veps^M) \Vert_{L^2((0,T)\times \oemf)}  + \frac{1}{\epsilon^{\frac32}} \Vert D(\ueps)\Vert_{L^{\infty}((0,T),L^2(\oems))} \le C
\end{align*}
From the Korn-inequality in the bulk domains (which constant is of course independent of $\epsilon$) we get
\begin{align*}
\Vert \nabla \veps^{\pm} \Vert_{L^2((0,T)\times \oepm)} \le C \Vert D(\veps^{\pm})\Vert_{L^2((0,T)\times \oepm)} \le C.
\end{align*}
Further, from the Korn-inequality in the thin perforated layer in Lemma \ref{KornInequalityPerforatedLayer} in the appendix, we obtain for the fluid velocity in the  layer
\begin{align*}
\frac{1}{\sqrt{\epsilon}} \sum_{i,j=1}^2 \Vert \partial_i \veps^{j,M} \Vert_{L^2((0,T)\times \oemf)} + \sqrt{\epsilon} \Vert \nabla \veps^M \Vert_{L^2((0,T)\times \oemf)} \le \frac{C}{\sqrt{\epsilon}} \Vert D(\veps^M)\Vert_{L^2((0,T)\times \oemf)} \le C.
\end{align*}
And for the the displacement we obtain the desired result by using again the Korn-inequality in Lemma \ref{KornInequalityPerforatedLayer}.
\\

\noindent\textbf{Step 2: }\textit{(Estimate for the time derivatives $\partial_t \ueps $ and $\partial_{tt} \ueps$)} We   differentiate  $\eqref{EQ:Var_Micro}$ with respect to time and choose in this equation as a test-function   $\phi = \partial_t \veps$ in $\oef$ and $\phi = \partial_{tt} \ueps$ in $\oems$. We get almost everywhere in $(0,T)$
\begin{align*}
\sum_{\pm}\frac12 &\frac{d}{dt}\Vert \partial_t \veps^{\pm} \Vert_{L^2(\oepm)}^2 + \frac{1}{2\epsilon}\frac{d}{dt} \Vert \partial_t \veps^M \Vert_{L^2(\oemf)}^2 + \frac{1}{2\epsilon}\frac{d}{dt} \Vert \partial_{tt} \ueps \Vert_{L^2(\oems)}^2
\\
&+\sum_{\pm} \Vert D(\partial_t \veps^{\pm} ) \Vert_{L^2(\oepm)}^2 + \foe \Vert D(\partial_t \veps^M ) \Vert_{L^2(\oemf)}^2 +\frac{1}{2\epsilon^3} \frac{d}{dt} \int_{\oems} A_{\epsilon} D(\partial_t \ueps) : D(\partial_t \ueps) dx
\\
&=\sum_{\pm} \int_{\oepm}\partial_t f_{\epsilon}^{\pm} \cdot  \partial_t \veps^{\pm} dx + \foe \int_{\oemf} \partial_t f_{\epsilon}^M \cdot  \partial_t  \veps^M dx
\\
&\le C \left( 1 + \sum_{\pm}\Vert \partial_t \veps^{\pm} \Vert_{L^2(\oepm)}^2 + \frac{1}{\sqrt{\epsilon}}\Vert \partial_t \veps^M \Vert_{L^2(\oemf)}^2 \right).
\end{align*}
Arguing as in $\eqref{Step1Inequality}$, we obtain for almost every $t \in (0,T)$ 
\begin{align}
\begin{aligned}
\label{EstimateHigherTimeDerivativeAuxiliaryInequality}
\sum_{\pm}\frac12 &\Vert \partial_t \veps^{\pm}(t)\Vert_{L^2(\oepm)}^2 + \frac{1}{2\epsilon} \Vert \partial_t \veps^M(t) \Vert_{L^2(\oemf)}^2 + \frac{1}{2\epsilon} \Vert \partial_{tt} \ueps(t) \Vert_{L^2(\oems)}^2
\\
&+ \sum_{\pm}\Vert D(\partial_t \veps^{\pm})\Vert_{L^2((0,t)\times \oepm)}^2 + \foe \Vert D(\partial_t \veps^M)\Vert_{L^2((0,t)\times \oemf)}^2 + \frac{c_0}{2\epsilon^3} \Vert D(\partial_t \ueps)(t)  \Vert_{L^2(\oems)}^2
\\
\le&  C \left( 1 + \sum_{\pm}\Vert \partial_t \veps^{\pm} \Vert_{L^2((0,t)\times \oepm)}^2 + \frac{1}{\sqrt{\epsilon}}\Vert \partial_t \veps^M \Vert_{L^2((0,t)\times \oemf)}^2 \right)
\\
+& \frac12 \left(\sum_{\pm}\Vert \partial_t \veps^{\pm}(0)\Vert_{L^2(\oepm)}^2 + \foe \Vert \partial_t \veps^M(0)\Vert_{L^2(\oemf)}^2 + \foe \Vert \partial_{tt} \ueps(0)\Vert_{L^2(\oems)}^2 + \frac{c_0}{\epsilon^3} \Vert D(\partial_t \ueps(0))\Vert_{L^2(\oems)}^2 \right).
\end{aligned}
\end{align}
We emphasize that due to the assumptions on the data  $\partial_t \ueps \in H^1((0,T),H^1(\oems))^3$ (not necessarily uniformly bounded with respect to $\epsilon$), and therefore $\partial_t \ueps \in C^0([0,T],H^1(\oems))^3$ with $D(\partial_t \ueps(0)) = 0$.
We have to estimate the initial terms for the time derivatives on the right-hand side. For this we evaluate for $\phi \in H^1(\oe)^3$ with $\phi = 0 $ on $\partial_D \oe$ the equation $\eqref{EQ:Var_Micro}$ in $t=0$, what is possible since the microscopic solution is regular enough. This can be shown by using similar arguments as in \cite[Section 27]{WlokaEnglisch}. We obtain 
 (with $\ueps(0) = 0$ and the assumption \ref{AssumptionInitialConditionsFluid}) 
\begin{align*}
\sum_{\pm}\int_{\oepm}& \partial_t \veps^{\pm}(0) \phi dx + \foe \int_{\oemf} \partial_t \veps^M (0) \phi dx  + \foe \int_{\oems} \partial_{tt} \ueps(0) \phi dx 
\\
=& -\sum_{\pm}\int_{\oepm}  D( \veps^{0,\pm}) : D( \phi) dx  - \foe \int_{\oemf} D( \veps^{0,M}) : D( \phi) dx 
\\
&+ \sum_{\pm}\int_{\oepm} \peps^{0,\pm} \nabla \cdot \phi dx + \foe \int_{\oemf} \peps^{0,M} \nabla \cdot \phi dx 
\\
&+ \sum_{\pm}\int_{\oepm} \feps^{\pm}(0) \phi dx + \foe \int_{\oemf} \feps^M(0) \phi dx 
\\
=& \sum_{\pm}\int_{\oepm} [\feps^{\pm}(0) - F_{\epsilon}^{0,\pm}] \phi dx + \foe \int_{\oemf} [\feps^M(0) - F_{\epsilon}^{0,M} ] \phi dx.
\end{align*}
By density this equation is valid for all $\phi \in L^2(\oe)^3$ and we obtain 
\begin{align*}
\partial_t \veps^{\pm}(0) &= \feps^{\pm}(0) - F_{\epsilon}^{0,\pm},
\\
\partial_t \veps^M(0) &= \feps^M(0) - F_{\epsilon}^{0,M},
\\
\partial_{tt} \ueps(0) &= 0.
\end{align*}
Since the $L^2$-norms of the functions on the right-hand side are bounded, due to the assumptions on the data, we obtain that the terms including the initial values on right-hand side in $\eqref{EstimateHigherTimeDerivativeAuxiliaryInequality}$ are bounded by a constant independent of $\epsilon$. 
Hence, we obtain with the Gronwall-inequality
\begin{align*}
\sum_{\pm}\Vert &\partial_t \veps^{\pm} \Vert_{L^{\infty}((0,T),L^2(\oepm))} + \frac{1}{\sqrt{\epsilon}} \Vert \partial_t \veps^M \Vert_{L^{\infty}((0,T),L^2(\oemf))} + \frac{1}{\sqrt{\epsilon}} \Vert \partial_{tt} \ueps \Vert_{L^{\infty}((0,T),L^2(\oems))} 
\\
&+ \sum_{\pm}\Vert D(\partial_t \veps^{\pm})\Vert_{L^2((0,T)\times \oepm)} + \frac{1}{\sqrt{\epsilon}} \Vert D(\partial_t  \veps^M) \Vert_{L^2((0,T)\times \oemf)}  + \frac{1}{\epsilon^{\frac32}} \Vert D( \partial_t \ueps)\Vert_{L^{\infty}((0,T),L^2(\oems))} \le C.
\end{align*}
Using again the Korn-inequality (keeping in mind that $\partial_t \ueps = 0 $ on $\partial_D \oems$), we obtain the estimate for the displacement $\ueps$.

\noindent\textbf{Step 3: }\textit{(Estimate for the bulk pressure $\peps^{\pm}$)} There exists $ \phi_{\epsilon}^{\pm} \in H^1(\oepm)^3$ with $\phi_{\epsilon}^{\pm} = 0$ on $\partial \oepm \setminus \partial_N \oef$, such that
\begin{align*}
\nabla \cdot \phi_{\epsilon}^{\pm} = -\peps^{\pm}, \qquad \Vert \phi_{\epsilon}^{\pm} \Vert_{H^1(\oepm)} \le C \Vert \peps^{\pm}\Vert,
\end{align*}
with a constant $C\gr 0 $ independent of $\epsilon$, see for example \cite[Proof of Theorem 5.4]{fabricius2017stokes} for more details. We extend the function $\phi_{\epsilon}^{\pm}$ by zero to the whole domain $\Omega_{\epsilon}$, which is an element of $H^1(\Omega_{\epsilon})^3$ vanishing on $\partial_D \oef$ and therefore an admissible test-function for the weak equation $\eqref{EQ:Var_Micro}$. We obtain with the estimates for $\partial_t \veps^{\pm}$ and $D(\veps^{\pm})$ already obtained
\begin{align*}
\Vert \peps^{\pm} \Vert_{L^2(\oepm)}^2 &= - \int_{\oepm} \partial_t \veps^{\pm} \phi_{\epsilon}^{\pm} dx  - \int_{\oepm} D(\veps^{\pm}) : D(\phi_{\epsilon}^{\pm}) dx  + \int_{\oepm} f_{\epsilon}^{\pm} \phi_{\epsilon}^{\pm} dx 
\\
&\le C  \Vert \phi_{\epsilon}^{\pm} \Vert_{H^1(\oepm)} \le C\Vert \peps^{\pm} \Vert_{L^2(\oepm)}.
\end{align*} 
\\
\noindent{\textbf{Step 4:} }\textit{(Estimate for the membrane pressure $\peps^M$)} We first construct a function with divergence equal to $\peps^M$. For $k \in K_{\epsilon}$ we define 
\begin{align*}
\peps^k: Z^f \rightarrow \R, \qquad \peps^k(y)= \peps(\epsilon(y+ k)).
\end{align*}
There exists a function $\phieps^k \in H^1(Z^f)^3$ with $\phieps^k = 0$ on $\partial Z^f \setminus S^{\pm}$ and 
\begin{align*}
\nabla_y \cdot \phieps^k = \peps^k, \qquad \Vert \phieps^k \Vert_{H^1(Z^f)} \le C \Vert \peps \Vert_{L^2(Z^f)}.
\end{align*}
We extend $\phieps^k$ to the whole by zero to the whole cell $Z$. Now, we define
\begin{align*}
\phieps: \oem \rightarrow \R^3,\qquad \phieps(x)= \epsilon \phieps^k\left(\fxe - k\right) \mbox{ for } x\in \epsilon(Z + k).
\end{align*}
Obviously, we have in $\oemf$
\begin{align*}
\nabla \cdot \phieps = \peps^M,
\end{align*}
and an elemental calculation shows
\begin{align*}
\foe \Vert \phieps \Vert_{L^2(\oem)} + \Vert \nabla \phieps \Vert_{L^2(\oem)} \le C \Vert \peps^M \Vert_{L^2(\oemf)}.
\end{align*}
By mirroring we extend the function $\phieps$ (with the same notation) to $\Omega_{2\epsilon}^M$, hence we have
\begin{align*}
\foe \Vert \phieps \Vert_{L^2(\Omega_{2\epsilon}^M)} + \Vert \nabla \phieps \Vert_{L^2(\Omega_{2\epsilon}^M)} \le C \Vert \peps^M \Vert_{L^2(\oemf)}.
\end{align*}
We emphasize that $\phieps$ has zero boundary-condition on the lateral boundary. Now we choose a cut-off function $\psi_{\epsilon} \in C_0^{\infty}([0,\epsilon))$ with $0\le \psi_{\epsilon} \le 1$, $\psi_{\epsilon}(0)=1$, and $\Vert \psi_{\epsilon}'\Vert_{L^{\infty}(0,\epsilon)} \le \frac{C}{\epsilon}$, and define the function 
\begin{align*}
\tphieps(x):= \begin{cases}
\phieps(x) \psi_{\epsilon}(x_n - \epsilon) &\mbox{ in } \Omega_{\epsilon}^+,
\\
\phieps(x) &\mbox{ in } \oem,
\\
\phieps(x)\psi_{\epsilon}(x_n + \epsilon) &\mbox{ in } \Omega_{\epsilon}^-.
\end{cases}
\end{align*}
This is an admissible test-function for $\eqref{EQ:Var_Micro}$ which vanishes on the solid part $\oems$ of the membrane. Especially, we have
\begin{align*}
\foe \Vert \tphieps \Vert_{L^2(\oepm)} + \Vert \nabla \tphieps\Vert_{L^2(\oepm)} \le C \Vert \peps^M \Vert_{L^2(\oemf)}. 
\end{align*}
Plugging in $\tphieps$ in $\eqref{EQ:Var_Micro}$ we obtain (with the estimates already obtained for $\ueps$ and $\veps$)
\begin{align*}
\foe \Vert \peps^M \Vert^2_{L^2(\oemf)} =& \sum_{\pm}\int_{\oepm} \partial_t \veps^{\pm} \cdot \tphieps dx + \foe \int_{\oemf} \partial_t \veps^M \cdot \tphieps dx +\sum_{\pm} \int_{\oepm} D(\veps^{\pm}): D(\tphieps) dx 
\\
&+ \foe \int_{\oemf} D(\veps^M) : D(\tphieps) dx -\sum_{\pm} \int_{\oepm} \peps^{\pm} \nabla \cdot \tphieps dx 
\\
&- \sum_{\pm}\int_{\oepm} f_{\epsilon}^{\pm} \cdot \tphieps dx - \foe \int_{\oemf} f_{\epsilon}^M \cdot \tphieps dx
\\
\le & C\left(\frac{1}{\sqrt{\epsilon}} \Vert \tphieps\Vert_{H^1(\oemf)} + \Vert \tphieps\Vert_{H^1(\oepm)} \right)
\le  \frac{C}{\sqrt{\epsilon}} \Vert \peps^M \Vert_{L^2(\oemf)}.
\end{align*}
\end{proof}

\section{Compactness results for the microscopic solution}
\label{SectionConvergenceMicroSolution}

In this section we derive the compactness results stated in Theorem \ref{MainResult} for the microscopic solution $(\veps,\peps,\ueps)$ for $\epsilon \to 0$, which then are the basis for the derivation of the macroscopic model. The starting point for these convergences are the \textit{a priori} estimates in Lemma \ref{AprioriEstimatesLemma}. While in the bulk domains we can work with usual convergence in $L^2$-spaces, in the thin perforated layer we work with the two-scale convergence for thin structures to deal with the homogenization and the dimension reduction for $\epsilon\to 0$. The definition of the two-scale convergence together with some important compactness results are summarized in the Appendix \ref{AppendixTwoScaleConvergence}.

\subsubsection*{Convergence of the bulk functions}

We start with the convergence of the fluid in the bulk domains, which we can treat with standard weak and strong compactness results in Sobolev spaces.

\begin{proposition}\label{ConvergenceBulkProp}
There exist $v_0^{\pm} \in L^2((0,T),H^1(\Omega^{\pm}))^3 \cap H^1((0,T),L^2(\Omega^{\pm}))^3$ with $\nabla \cdot v_0^{\pm} = 0$, and $p_0^{\pm} \in L^2((0,T)\times \Omega^{\pm})$, such that up to a subsequnce  for every $\beta \in \left(\frac12,1\right)$ 
\begin{align*}
\veps^{\pm}(\cdot_t , \cdot_x \pm \epsilon e_3) &\rightarrow v_0^{\pm} &\mbox{ strongly in }& L^2((0,T),H^{\beta}(\Omega^{\pm}))^3,
\\
\nabla \veps^{\pm}(\cdot_t , \cdot_x \pm \epsilon e_3) &\rightharpoonup \nabla v_0^{\pm} &\mbox{ weakly in }& L^2((0,T),L^2(\Omega^{\pm}))^{3\times 3},
\\
\partial_t \veps^{\pm}(\cdot_t , \cdot_x \pm \epsilon e_3) &\rightharpoonup \partial_t v_0^{\pm} &\mbox{ weakly in }& L^2((0,T),L^2(\Omega^{\pm}))^3,
\\
\peps^{\pm}(\cdot_t , \cdot_x \pm \epsilon e_3) &\rightharpoonup p_0^{\pm} &\mbox{ weakly in }& L^2((0,T)\times \Omega^{\pm}).
\end{align*}
\end{proposition}
\begin{proof}
This is a direct consequence of the \textit{a priori} estimates in Lemma \ref{AprioriEstimatesLemma} and the Aubion-Lions-Lemma, see \cite{Lions}.
\end{proof}

\subsubsection*{Convergence for the displacement}

The displacement of the elastic structure in the thin layer has a different behavior in the limit $\epsilon \to 0$ in tangential and vertical direction.
%
%
%
More precisely, the two-scale limit fulfills a Kirchhoff-Love displacement.  Usually, two-scale compactness results based on \textit{a priori} estimates including the gradient include in the scale limit the zeroth- and first-order term of the formal asymptotic expansion. However, in our case,  the bound of the symmetric gradient  from Lemma \ref{AprioriEstimatesLemma} (which is one order higher than the gradient) guarantees that the two-scale limit of $\epsilon^{-1}D(\ueps)$ involves a corrector term of second order.

\begin{proposition}\label{ConvergenceResultsDisplacementProp}
There exists $u_0^3 \in H^1((0,T),H_0^2(\Sigma))\cap H^2((0,T),L^2(\Sigma))$ and $\tilde{u}_1 \in H^1((0,T),H^1_0(\Sigma))^3$ with $\tilde{u}_1^3 = 0$, and $u_2 \in H^1((0,T), L^2( \Sigma, H_{\#}^1(Z)/\R))^3$, such that up to a subsequence (for $\alpha = 1,2$)
\begin{align*}
\chi_{\oems} \frac{\ueps^{\alpha}}{\epsilon} &\rats \chi_{Z^s} (\tilde{u}_1^{\alpha} - y_3 \partial_{\alpha} u_0^3),
\\
\chi_{\oems} \ueps^3 &\rats \chi_{Z^s} u_0^3,
\\
\foe \chi_{\oems} D(\ueps) &\rats \chi_{Z^s} \big(D_{\x}(\tilde{u}_1) - y_3 \nabla_{\x}^2 u_0^3 + D_y(u_2)\big).
\end{align*}
The same convergence results are valid if we replace $\ueps $ with $\partial_t \ueps$ and the limit functions with their time derivatives.
For the second time derivative we have for a subsequence  
\begin{align*}
\chi_{\oems} \partial_{tt} \ueps^3 &\rats \chi_{Z^s} \partial_{tt} u_0^3.
\end{align*}
 Further, it holds up to a subsequence that
\begin{align*}
\ueps\vert_{\geps} &\rats (0,0,u_0^3)^T \qquad\mbox{in the two-scale sense on } \geps,\\
\partial_t \ueps\vert_{\geps} &\rats (0,0,\partial_t u_0^3)^T \qquad\mbox{in the two-scale sense on } \geps.
\end{align*}
\end{proposition}

\begin{proof}
The convergence results in the thin layer follow directly from Lemma \ref{AprioriEstimatesLemma} and the two-scale compactness results from Lemma \ref{MainTheoremTwoScaleConvergence} in the appendix.
For the result on the surface we use the well known trace-inequality (obtained by a simple decomposition argument), to obtain for $i=1,2,3$
\begin{align*}
\Vert \ueps^i\Vert_{L^2((0,T)\times \geps)} \le C\left(\frac{1}{\sqrt{\epsilon}} \Vert \ueps^i \Vert_{L^2((0,T)\times \oems)} + \sqrt{\epsilon} \Vert \nabla \ueps^i \Vert_{L^2((0,T)\times \oems)} \right) \le C.
\end{align*}
We emphasize that for $i=1,2$ the norm of $\ueps^i$ above is even of order $\epsilon$, 
which, however, does not really simplify the following argumentation.
Due to Lemma \ref{LemmaBasicTSCompactness} in the appendix, there exists $u_0^{\Gamma} \in L^2((0,T)\times \Sigma \times \Gamma)^3$, such that up to a subsequence
\begin{align*}
\ueps\vert_{\geps} \rats u_0^{\Gamma} \qquad\mbox{  on } \geps.
\end{align*}
Further, for all $\phi \in C_0^{\infty}((0,T)\times \Sigma , C_{\#}^{\infty}(\overline{Z}))^3$ with $\phi = 0$ on $S^{\pm}$ it holds that $(u_0 = (0,0,u_0^3)^T$)
\begin{align*}
0 &= \lim_{\epsilon\to 0} \foe \int_0^T \int_{\oems} \epsilon \nabla \ueps : \phi\left(t,\x,\fxe\right) dx  dt
\\
&= \lim_{\epsilon\to 0}\bigg\{- \foe \int_0^T \int_{\oems} \ueps \cdot \left[ \epsilon\nabla_{\x} \cdot \phi \left(t,\x,\fxe\right) + \nabla_y \cdot \phi\left(t,\x,\fxe\right) \right] dx
\\
&+ \int_0^T \int_{\geps} \ueps \cdot \left[\phi\left(t,\x,\fxe\right)\nu \right] d\sigma dt\bigg\}
\\
&= - \int_0^T \int_{\Sigma} \int_{Z^s} u_0 \cdot \left[\nabla_y \cdot \phi(t,\x,y) \right] dy d\x dt + \int_0^T \int_{\Sigma}\int_{\Gamma} u_0^{\Gamma} \cdot \left[\phi(t,\x,y) \nu\right]d\sigma_y d\x dt
\\
&= \int_0^T \int_{\Sigma} \int_{\Gamma} \left(u_0 - u_0^{\Gamma}\right) \cdot \left[ \phi(t,\x,y) \nu\right] d\sigma_y d\x dt.
\end{align*}
By a density argument and the surjectivity of the normal-trace operator we obtain $u_0 = u_0^{\Gamma}$. In a similar way we show the result for $\partial_t \ueps$.
\end{proof}

\begin{remark}
The function $u_2$ is only unique up to a rigid-displacement (depending on $(t,\x)$). However, the only $Y$-periodic rigid-dispacements are constants.
\end{remark}

\subsubsection*{Convergence for the fluid-velocity in the membrane}

 As can be seen from Lemma \ref{AprioriEstimatesLemma}, the estimates for the fluid velocity in the thin layer have a different scaling than those for the displacement. Thus, we cannot apply the compactness result in Lemma \ref{MainTheoremTwoScaleConvergence} to determine the two-scale limit of the velocity. However, by using the continuity of the fluid and solid velocity on $\geps$, we show that in the limit $\epsilon \to 0$, the velocity of the fluid in the thin layer behaves like the velocity of the solid.

\begin{proposition}\label{ConvergenceFluidMembraneProp}
Let $\tveps^M$ be the extension of $\veps^M$ from Lemma \ref{TheoremExtensionOperator}.
We have up to a subsequence
\begin{align*}
\chi_{\oemf}\tveps^M &\rats  \chi_{Z^f}(0,0,\partial_t u_0^3)^T.
\end{align*}
Especially, the following convergence results hold (up to a subsequence)
\begin{align*}
\chi_{\oemf}\veps^M &\rats \chi_{Z^f} (0,0,\partial_t u_0^3)^T,
\\
\chi_{\oemf}\partial_t \veps^M &\rats \chi_{Z^f}(0,0,\partial_{tt} u_0^3)^T.
\end{align*}
Further, the following interface condition holds
\begin{align*}
v_0^{\pm} = (0,0,\partial_t u_0^3)^T \qquad\mbox{ on } (0,T)\times \Sigma.
\end{align*}
\end{proposition}

\begin{proof}
The \textit{a priori} estimates in Lemma \ref{AprioriEstimatesLemma} and the estimates from Lemma \ref{TheoremExtensionOperator} for the extension $\tveps^M$, together with the two-scale compactness result in Lemma \ref{LemmaBasicTSCompactness}, imply the existence of $v_0^M \in L^2((0,T)\times \Sigma, H^1_{\#}(Z))^3$ with $\partial_t \left(\chi_{Z^f}v_0^M \right)\in L^2((0,T)\times \Sigma \times Z^f)^3$, and $\xi \in L^2((0,T)\times \Sigma \times Z)^{3\times 3}$ such that up to a subsequence
\begin{align*}
\tveps^M &\rats v_0^M,
\\
\epsilon \nabla \tveps^M &\rats \nabla_y v_0^M,
\\
\chi_{\oemf}\partial_t \veps^M &\rats \partial_t\left(\chi_{Z^f} v_0^M\right),
\\
D(\tveps^M) &\rats \xi.
\end{align*}
Especially, we obtain $D_y(v_0^M) = 0$. Hence, $v_0^M$ is a rigid-displacement with respect to $y$. Due to the periodicity of $v_0^M$ it follows that $v_0^M(t,\x,y) = v_0^M(t,\x)$ with $v_0^M \in L^2((0,T)\times \Sigma)^3$. Due to the boundary condition $\partial_t \ueps = \veps^M$ on $\geps$ and Proposition \ref{ConvergenceResultsDisplacementProp} we obtain
\begin{align*}
\veps^M\vert_{\geps} \rats (0,0,\partial_t u_0^3)^T \qquad\mbox{ on } \geps.
\end{align*}
In a similar way as in the proof of Proposition \ref{ConvergenceResultsDisplacementProp} we obtain $v_0^M = (0,0,\partial_t u_0^3)^T$.
Especially, we obtain
\begin{align*}
\veps^M\vert_{\seps^{\pm}} \rats (0,0,\partial_t u_0^3)^T \qquad\mbox{ on } \seps^{\pm}.
\end{align*}
Here the two-scale convergence on $\seps^{\pm}$ is the usual two-scale convergence in $\R^{n-1}$, see \cite{Allaire_TwoScaleKonvergenz}.
Now, we prove the interface condition for $v_0^{\pm}$ on $\Sigma$. Since $\veps^{\pm} = \veps^M$ on $\seps^{\pm}$, we obtain with Proposition \ref{ConvergenceBulkProp} for all $\phi \in C_0^{\infty}((0,T)\times \Sigma, C_{\per}^{\infty}(Y))^3$
\begin{align*}
\int_0^T \int_{\Sigma} \int_Y v_0^{\pm}(t,x) \cdot \phi(t,\x,\y) d\y d\x dt &=\lim_{\epsilon\to 0} \int_0^T \int_{\seps^{\pm}} \veps^{\pm}(t,x) \cdot \phi\left(t,\x,\dfrac{\x}{\epsilon}\right) d\sigma dt
\\
&=\lim_{\epsilon\to 0} \int_0^T \int_{\seps^{\pm}} \veps^M(t,x)\cdot \phi \left(t,\x,\dfrac{\x}{\epsilon}\right) d\sigma dt
\\
&= \int_0^T \int_{\Sigma}\int_Y \partial_t u_0(t,\x) \cdot \phi(t,\x,\y ) d\y d\x dt.
\end{align*}
This implies the desired result.
\end{proof}
In summary, we proved the convergence results in Theorem \ref{MainResult}.

\section{Derivation of the macroscopic model}
\label{SectionDerivationMacroModel}

To finish the proof of the main result in Theorem \ref{MainResult}, we have to show the the limit functions $(v_0^{\pm},p_0^{\pm},\tilde{u}_1,u_0^3)$ from Section \ref{SectionConvergenceMicroSolution} is the unique weak solution of the macroscopic model $\eqref{MacroModelStrongFormulation}$. We start with the derivation of the cell problems which enter in the definition of the homogenized elasticity tensors.
We define the symmetric matrices $M_{ij} \in \R^{3\times 3}$ for $i,j=1,2,3$ by
\begin{align*}
M_{ij} = \frac{e_i \otimes e_j}{2} + \frac{e_j \otimes e_i}{2}.
\end{align*}
Further, we define  $\chi_{ij} \in H^1_{\#}(Z^s)^3$ as the solutions of the cell problems 
\begin{align}
\begin{aligned}\label{CellProblemStandard}
-\nabla_y \cdot (A (D_y(\chi_{ij}) + M_{ij})) &= 0 &\mbox{ in }& Z^s,
\\
-A(D_y(\chi_{ij}) + M_{ij} ) \nu &= 0 &\mbox{ on }& \Gamma,
\\
\chi_{ij} \mbox{ is } Y\mbox{-periodic, } & \int_{Z^f} \chi_{ij} dy = 0.
\end{aligned}
\end{align}
Due to the Korn-inequality, this problem has a unique weak solution. We emphasize again that the only rigid-displacements on $Z^s$, which are $Y$-periodic, are constants.

Additionally, we define $\chi_{ij}^B \in H^1_{\#}(Z^s)^3$ as the solutions of the cell problems
\begin{align}
\begin{aligned}\label{CellProblemHesse}
-\nabla_y \cdot \left( A(D_y(\chi_{ij}^B) - y_3 M_{ij})\right) &= 0 &\mbox{ in }& Z^s,
\\
-A(D_y(\chi_{ij}^B) - y_3 M_{ij})\nu &= 0 &\mbox{ on }& \Gamma,
\\
\chi_{ij}^B \mbox{ is } Y\mbox{-periodic, } & \int_{Z^f} \chi_{ij}^B dy = 0.
\end{aligned}
\end{align}
In the same way as above we obtain the existence of a unique weak solution.

\begin{proposition}\label{PropositionCorrectorU2Representation}
The limit function $u_2$ from Proposition \ref{ConvergenceResultsDisplacementProp} fulfills
\begin{align*}
u_2(t,\x,y) = \sum_{i,j=1}^2 \left[ D_{\x}(\tilde{u}_1)_{ij} (t,\x) \chi_{ij}(y) +  \partial_{ij}u_0^3(t,\x) \chi_{ij}^B(y) \right],
\end{align*}
where the cell solutions $\chi_{ij}$ and $\chi_{ij}^B$ are defined in $\eqref{CellProblemStandard}$ and $\eqref{CellProblemHesse}$.
\end{proposition}
\begin{proof}
Let $\phi \in C_0^{\infty}((0,T)\times \Sigma , C_{\#}^{\infty}(\overline{Z}))^3$ with $\phi = 0 $ on $S^{\pm}$. As a test-function in $\eqref{EQ:Var_Micro}$ we choose
\begin{align*}
\phieps(t,x):= \begin{cases} \epsilon^2 \phi\left(t,\x,\fxe\right) &\mbox{ in } \oem,
\\
0 &\mbox{ in } \oepm,
\end{cases}
\end{align*}
to obtain
\begin{align*}
\epsilon &\int_{\oemf} \partial_t \veps^M (t,x) \phi\tbxfxe dx + \epsilon \int_{\oems} \partial_{tt}\ueps (t,x) \phi\tbxfxe dx
\\
&+ \foe \int_{\oemf} D(\veps^M): \left[\epsilon^2 D_{\x}(\phi)\tbxfxe + \epsilon D_y(\phi)\tbxfxe \right] dx 
\\
&+ \foe \int_{\oems} A_{\epsilon} D(\ueps)(t,x) : \left[ D_{\x}(\phi)\tbxfxe + \foe D_y(\phi) \tbxfxe \right] dx
\\
&- \foe \epsilon \peps^M(t,x) \left[ \epsilon \nabla_{\x} \cdot \phi\tbxfxe + \nabla_y \phi\tbxfxe \right] dx 
\\
=& \epsilon \int_{\oemf} f_{\epsilon}^M \cdot \phi \tbxfxe dx.
\end{align*}
Based on the \textit{a priori} estimates from Lemma \ref{AprioriEstimatesLemma} it is easy to check that all terms in the equation above, excepting the one including $D_y(\phi)$ are of order $\epsilon$. Thus, using the convergence result for $\epsilon^{-1}D(u_\epsilon) $ from Section \ref{SectionConvergenceMicroSolution}, we obtain for $\epsilon \to 0$, after an integration with respect to time, that
\begin{align*}
0 = \int_0^T \int_{\Sigma} \int_{Z^s} A \left[D_{\x}(\tilde{u}_1)(t,\x) - y_3 \nabla_{\x}^2 u_0^3(t,\x) + D_y(u_2)(t,\x,y)\right] : D_y(\phi)(t,\x,y) dy d\x dt.
\end{align*}
In other words, $u_2$ is a weak solution of the problem
\begin{align*}
- \nabla_y \cdot \left( A(D_{\x}(\tilde{u}_1) - y_3 \nabla_{\x}^2 u_0^3 + D_y(u_2))\right) &= 0 &\mbox{ in }& (0,T)\times \Sigma \times Z^s,
\\
- A(D_{\x}(\tilde{u}_1) - y_3 \nabla_{\x}^2 u_0^3 + D_y(u_2))\nu &= 0 &\mbox{ on }& (0,T)\times \Sigma \times \Gamma,
\\
u_2 \mbox{ is } Y\mbox{-periodic, }& \int_{Z^s} u_2 dy = 0.
\end{align*}
For given $(\tilde{u}_1, u_0^3)$, this problem has a unique solution $u_2$, due to the Korn-inequality and the Lax-Milgram-Lemma. An elemental calculation gives the desired result.
\end{proof}

\begin{remark}\label{RemarkZellprobleme}
The result is still valid if $Z^s$ touches the upper boundary $S^{\pm}$ of $Z$.  In this case we choose in the proof  $\phi \in C_0^{\infty}((0,T)\times \Sigma , C_{\#}^{\infty}(\overline{Z}))^3$ (without zero-boundary conditions on $S^{\pm}$). We extend this function smoothly to $Y\times (-2,2)$ with respect to $y$, such that $\phi = 0$ for $\vert y_n\vert \gr \frac32$. As a test-function we choose in $\eqref{EQ:Var_Micro}$ the function $\phieps(t,x) = \epsilon^2 \phi \tbxfxe$. This leads to additional terms in the bulk domains of the form (we only consider the term including the spatial derivatives, since the other terms can be treated in a simpler way)
\begin{align*}
\int_0^T \int_{\oepm} D(\veps^{\pm}) : \left[\epsilon^2 D_{\x}(\phi)\tbxfxe + \epsilon D_y(\phi)\tbxfxe \right] dx dt.
\end{align*}
Obviously, this term is of order $\epsilon$ (even $\epsilon^{\frac32}$, see the proof of Proposition \ref{PropV1M} below). Hence, we obtain the same cell problem for $u_2$.
\end{remark}


To finish the proof of Theorem \ref{MainResult}, we have to show that $(v_0^{\pm},p_0^{\pm},\tilde{u}_1,u_0^3)$ is a weak solution of the macro-model $\eqref{MacroModelStrongFormulation}$, and that this solution is unique. We start with the construction of  a test-function for the microscopic equation $\eqref{EQ:Var_Micro}$ adapted to the structure of the macroscopic model.
Let $\psi \in C_0^{\infty}([0,1)$ be a cut-off function with $0 \le \psi \le 1$ and $\psi(0)= 1$, $V \in \spaceH$,  $\bar{U} = (U_1,U_2) \in H_0^1(\Sigma)^2$. We define
\begin{align*}
\phieps(t,x):=  \begin{cases}
V(x) +\epsilon \psi\left(\frac{x_3 - \epsilon}{\epsilon}\right)\left( \begin{pmatrix}
U_1(\x) \\ U_2(\x) \\ 0 
\end{pmatrix}
-  \begin{pmatrix}
\partial_1 V^3 (\x) \\ \partial_2 V^3(\x) \\ 0 
\end{pmatrix}
\right)  &\mbox{ for } x \in \Omega_{\epsilon}^+,
\\
\begin{pmatrix}
0 \\ 0 \\ V^3(\x)
\end{pmatrix} 
+\epsilon \left( \begin{pmatrix}
U_1(\x) \\ U_2(\x) \\ 0 
\end{pmatrix}
- \frac{x_3}{\epsilon} \begin{pmatrix}
\partial_1 V^3 (\x) \\ \partial_2 V^3(\x) \\ 0 
\end{pmatrix}
\right) &\mbox{ for } x \in \oem,
\\
V(x) +\epsilon \psi\left(\frac{x_3 + \epsilon}{\epsilon}\right)\left( \begin{pmatrix}
U_1(\x) \\ U_2(\x) \\ 0 
\end{pmatrix}
- \begin{pmatrix}
\partial_1 V^3 (\x) \\ \partial_2 V^3(\x) \\ 0 
\end{pmatrix}
\right)  &\mbox{ for } x \in \Omega_{\epsilon}^-.
\end{cases}
\end{align*}
Here, $V^3(\x)$ is the trace of $V^3$ on $\Sigma$. We write $\phieps = (\phieps^+,\phieps^M,\phieps^-)$. We use  the notation $U= (\bar{U},0)^T$ and $\nabla_{\x} V^3 = (\partial_1 V^3,\partial_2 V^3,0)$. Obviously, it holds that
\begin{align*}
\nabla \cdot \phieps^{\pm} &= \nabla \cdot V + \epsilon \psi\left(\frac{x_3 - \epsilon}{\epsilon} \right) \left[ \nabla_{\x} \cdot \bar{U} - \Delta_{\x} V^3\right],
\\
\nabla \phieps^{\pm} &= \nabla V  + \epsilon \psi\left(\frac{x_3 \mp \epsilon}{\epsilon}\right) \left(\nabla_{\x} U - \nabla_{\x}^2 V^3 \right) + \psi'\left(\frac{x_3\mp \epsilon}{\epsilon}\right) (U - \nabla_{\x} V^3) \otimes e_3,
\\
\nabla \cdot \phieps^M &= \epsilon \left(\nabla_{\x} \cdot \bar{U} - \frac{x_3}{\epsilon} \Delta_{\x} V^3\right),
\\
\nabla \phieps^M &= \epsilon \left(\nabla_{\x}\bar{U} - \frac{x_3}{\epsilon} \nabla_{\x}^2 V^3\right).
\end{align*}
Plugging in $\phieps$ as a test-function in $\eqref{EQ:Var_Micro}$ and using the calculations above, we obtain (using that the Frobenius inner product between symmetric and skew-symmetric matrices is zero)
\begin{align}
\begin{aligned}\label{AuxiliaryEquationMacroModel}
\sum_{\pm}\int_{\oepm}& \partial_t \veps^{\pm} \cdot \left[V(x) + \epsilon \psi\left(\frac{x_3 \mp \epsilon}{\epsilon}\right) (U(\x) - \nabla_{\x} V^3(\x)) \right] dx
\\
+& \foe \int_{\oemf} \partial_t \veps^{3,M} V^3 + \epsilon \partial_t \veps^M \cdot \left[ U(\x) - \frac{x_3}{\epsilon} \nabla_{\x} V^3(\x) \right]  dx
\\
+& \foe \int_{\oems} \partial_{tt} \ueps^3 V^3 + \epsilon \partial_{tt} \ueps \cdot  \left[ U(\x) - \frac{x_3}{\epsilon} \nabla_{\x} V^3(\x) \right]  dx
\\
+& \sum_{\pm} \int_{\oepm} D(\veps^{\pm}) : \left[  \nabla V + \epsilon \psi\left(\frac{x_3 \mp \epsilon}{\epsilon}\right) \left(\nabla_{\x} U - \nabla_{\x}^2 V^3 \right) + \psi'\left(\frac{x_3\mp \epsilon}{\epsilon}\right) (U - \nabla_{\x} V^3) \otimes e_3 \right]dx
\\
+& \foe \int_{\oemf} D(\veps^M) : \epsilon \left(D_{\x}(\bar{U}) - \frac{x_3}{\epsilon} \nabla_{\x}^2 V^3\right) dx
\\
+& \frac{1}{\epsilon^3} \int_{\oems} A_{\epsilon} D(\ueps) : \epsilon \left(D_{\x}(\bar{U}) - \frac{x_3}{\epsilon} \nabla_{\x}^2 V^3\right) dx
\\
-& \sum_{\pm}\int_{\oepm} \peps^{\pm} \left[ \nabla \cdot V + \epsilon \psi\left(\frac{x_3 - \epsilon}{\epsilon} \right) \left[ \nabla_{\x} \cdot \bar{U} - \Delta_{\x} V^3\right]\right] dx 
\\
-& \foe \int_{\oemf} \peps^M \epsilon \left(\nabla_{\x} \cdot \bar{U} - \frac{x_3}{\epsilon} \Delta_{\x} V^3\right) dx
\\
&= \sum_{\pm}\int_{\oepm} f_{\epsilon}^{\pm} \cdot \left[ V + \epsilon \psi\left(\frac{x_3 \mp \epsilon}{\epsilon} \right) \left(U - \nabla_{\x} V^3\right)  \right] dx 
\\
+& \foe \int_{\oemf} f_{\epsilon}^{3,M} V^3 + f_{\epsilon}^M \cdot \epsilon \psi\left(\frac{x_3 - \epsilon}{\epsilon}\right) \left[U - \frac{x_3}{\epsilon} \nabla_{\x}V^3 \right]dx
\end{aligned}
\end{align}
We multiply this equation with $\eta \in C_0^{\infty}([0,T))$ and integrate with respect to time and pass to the limit $\epsilon \to 0$. The terms including $\psi'$ vanish, since we have
\begin{align*}
\bigg\vert \int_0^T & \int_{\Omega_{\epsilon}^+} \eta(t)D(\veps^{+}) : \psi'\left(\frac{x_3 - \epsilon}{\epsilon}\right) (U(\x)-\nabla_{\x} V^3(\x))\otimes e_3  dx dt \bigg\vert
\\
&\le C \Vert D(\veps^{+})\Vert_{L^2((0,T)\times \Omega_{\epsilon}^+)} \Vert U - \nabla_{\x} V^3\Vert_{L^2((0,T)\times \Sigma \times (\epsilon,2\epsilon))} \left\Vert \psi'\left(\frac{x_3 - \epsilon}{\epsilon} \right)\right\Vert_{L^{\infty}(\epsilon,2\epsilon)}
\\
&\le C \sqrt{\epsilon}.
\end{align*}
In the same way we can treat the terms including $D(\veps^-)$ and $\peps^{\pm}$.  Passing to the limit in $\eqref{AuxiliaryEquationMacroModel}$, after integrating with respect to time, we obtain 
\begin{align*}
\sum_{\pm}\int_0^T & \int_{\Omega^{\pm}} \partial_t v_0^{\pm} \cdot V \eta dx dt + \vert Z^f\vert\int_0^T \int_{\Sigma} \partial_{tt} u_0^3 V^3 \eta  d\x dt
+ \vert Z^s \vert \int_0^T \int_{\Sigma} \partial_{tt} u_0^3 V^3 \eta  d\x dt
\\
+&\sum_{\pm} \int_0^T \int_{\Omega^{\pm}} D(v_0^{\pm}) : D(V) \eta dx dt
- \int_0^T \int_{\Omega^{\pm}} p_0^{\pm} \nabla \cdot V \eta  dx dt
\\
 +& \int_0^T \int_{\Sigma} \int_{Z^s} A \left[ D_{\x}(\tilde{u}_1) - y_3 \nabla_{\x}^2 u_0^3 + D_y(u_2) \right] : \left[D_{\x}(U) - y_3 \nabla_{\x}^2 V^3 \right] \eta dxdt
\\
&= \sum_{\pm}\int_0^T \int_{\Omega^{\pm}} f_0^{\pm} \cdot V \eta dx dt + \int_0^T \int_{\Sigma} \int_{Z^f} f_0^{3,M} V^3 \eta dy d\x dt.
\end{align*}
Using the representation for $u_2$ and the tensors $a^{\ast},b^{\ast},c^{\ast} \in \R^{2\times 2 \times 2 \times 2}$ (see also  \cite{griso2020homogenization}) with components $\alpha,\beta,\gamma,\delta = 1,2$ defined by  
\begin{align}
\begin{aligned}
\label{HomogenizedTensors}
a^{\ast}_{\alpha\beta\gamma\delta} &:=  \frac{1}{\vert Z^s \vert} \int_{Z^s} A  \left(D_y(\chi_{\alpha \beta}) + M_{\alpha\beta}\right): \left(D_y (\chi_{\gamma\delta})  + M_{\gamma\delta} \right)dy,
\\
b^{\ast}_{\alpha\beta\gamma\delta} &:=   \frac{1}{\vert Z^s \vert} \int_{Z^s} A \left(D_y(\chi_{\alpha \beta}^B)  - y_3 M_{\alpha\beta} \right) : \left(D_y (\chi_{\gamma\delta})  + M_{\gamma\delta} \right)dy,
\\
c^{\ast}_{\alpha\beta\gamma\delta} &:=   \frac{1}{\vert Z^s \vert} \int_{Z^s} A  \left(D_y(\chi_{\alpha \beta}^B)  - y_3 M_{\alpha\beta} \right): \left(D_y (\chi_{\gamma\delta}^B)  -y_3 M_{\gamma\delta}\right)dy,
\end{aligned}
\end{align}
we obtain after an elemental calculation 
\begin{align}
\begin{aligned}
\label{MacroModelVarForm}
\sum_{\pm}&\int_0^T  \int_{\Omega^{\pm}} \partial_t v_0^{\pm} \cdot V \eta  dx dt +\int_0^T \int_{\Sigma} \partial_{tt} u_0^3 V^3 \eta d\x dt
\\
+& \sum_{\pm}\int_0^T \int_{\Omega^{\pm}} D(v_0^{\pm}) : D(V) \eta dx dt
- \sum_{\pm}\int_0^T \int_{\Omega^{\pm}} p_0^{\pm} \nabla \cdot V \eta dx dt
\\
+& \int_0^T \eta \int_{\Sigma} a^{\ast} D_{\x}(\tilde{u}_1) : D_{\x}(\bar{U}) + b^{\ast} \nabla_{\x}^2 u_0^3 : D_{\x}(\bar{U}) + b^{\ast} D_{\x}(\tilde{u}_1) : \nabla_{\x}^2 V^3 + c^{\ast} \nabla_{\x}^2 u_0^3 : \nabla_{\x}^2 V^3 d\x dt
\\
=& \sum_{\pm}\int_0^T \int_{\Omega^{\pm}} f_0^{\pm} \cdot V \eta dx dt + \int_0^T \int_{\Sigma} \int_{Z^f} f_0^{3,M} V^3 \eta dy d\x dt,
\end{aligned}
\end{align}
for all $V \in \spaceH$, $\bar{U} \in H_0^1(\Sigma)^2$ and $\eta \in C_0^{\infty}([0,T))$.  This gives the variational equation $\eqref{VarFormMacroModel}$ for the macro-model.

The initial conditions $\eqref{InitialConditionsMacroModel}$ are a consequence of the convergence results in Proposition \ref{ConvergenceBulkProp} and \ref{ConvergenceResultsDisplacementProp}. In fact, for all $\phi \in C_0^{\infty}([0,T)\times \Omega)$ it holds that
\begin{align*}
    \vert Z^s \vert &\int_0^T \int_{\Sigma} \partial_{tt}u_0^3 \phi d\x dt = \lim_{\epsilon\to 0} \foe  \int_0^T \int_{\oems} \partial_{tt}\ueps^3 \phi dx dt
    \\
    &= \lim_{\epsilon\to 0}\foe  \int_0^T \int_{\oems} \ueps^3 \phi'' dx dt
    =\vert Z^s \vert \int_0^T \int_{\Sigma} u_0^3 \phi'' d\x dt
    \\
    &=\vert Z^s \vert \int_0^T \int_{\Sigma} \partial_{tt}u_0^3\phi d\x dt + \vert Z^s \vert \int_{\Sigma} \partial_t u_0^3(0) \phi(0) d\x  -  \vert Z^s\vert \int_{\Sigma} u_0^3(0)\phi'(0) d\x.
\end{align*}
This implies $u_0^3(0) = \partial_t u_0^3 (0) =0$, and with similar arguments we get $v_0^{\pm}(0) = v^{0,\pm}$.

It remains to show the uniqueness of the macroscopic solution.
For this it is enough to show that $(v_0^{\pm},\tilde{u}_1,u_0^3) = 0$ if $ (v^{0,\pm},v^{3,0,M},f_0^{\pm},f_0^{3,M}) = 0$. If the latter is fulfilled we have from $\eqref{MacroModelVarForm}$ almost everywhere in $(0,T)$
\begin{align*}
0 = &\sum_{\pm} \int_{\Omega^{\pm}} \partial_t v_0^{\pm} \cdot V \eta  dx   +  \int_{\Sigma} \partial_{tt} u_0^3 V^3 d\x  
\\
+&\sum_{\pm}  \int_{\Omega^{\pm}} D(v_0^{\pm}) : D(V)   dx  
- \sum_{\pm} \int_{\Omega^{\pm}} p_0^{\pm} \nabla \cdot V dx  
\\
+&   \int_{\Sigma} a^{\ast} D_{\x}(\tilde{u}_1) : D_{\x}(\bar{U}) + b^{\ast} \nabla_{\x}^2 u_0^3 : D_{\x}(\bar{U}) + b^{\ast} D_{\x}(\tilde{u}_1) : \nabla_{\x}^2 V^3 + c^{\ast} \nabla_{\x}^2 u_0^3 : \nabla_{\x}^2 V^3 d\x  
\end{align*} 
for all $V \in \spaceH$ and $\bar{U} \in H^1_0(\Sigma)^2$. Choosing $V = v_0$ and $\bar{U} = \partial_t\tilde{u}_1$ we obtain (since the form induced by $a^{\ast}$, $b^{\ast}$, and $c^{\ast}$ is coercive, see \cite[Theorem 2]{griso2020homogenization}) for a constant $c_0^{\ast} \gr 0$
\begin{align*}
\sum_{\pm}\frac12 \frac{d}{dt} \Vert v_0^{\pm}\Vert_{L^2(\Omega^{\pm})}^2 + \frac12 \frac{d}{dt} \Vert \partial_t u_0^3 \Vert^2_{L^2(\Sigma)} +\frac{c_0^{\ast}}{2 }\frac{d}{dt}\left(\Vert D_{\x}(\tilde{u}_1)\Vert^2_{L^2(\Sigma)} + \Vert \nabla_{\x}^2 u_0^3 \Vert_{L^2(\Sigma)}^2 \right) \le 0.
\end{align*}
Integration with respect to time and using the Korn-inequality, we obtain the uniqueness for the macro-solution.

\begin{corollary}
All the convergence results for $\ueps$ and $\veps$ are valid for the whole sequence.
\end{corollary}

\section{Higher order correctors for the fluid in the membrane}
\label{SectionCorrectors}
In this section we identify a first order corrector for the fluid velocity and the zeroth order term for the fluid pressure in the membrane with respect to two-scale convergence. Here we assume that $\geps$ is a $C^{1,1}$ boundary, and therefore also $\Gamma \in C^{1,1}$.

\begin{lemma}\label{LemmaExistenceCorrectorV1M}
Let $\veps$ be the solution of the micro-model $\eqref{MicroscopicModel}$. Then there exists $v_1^M \in L^2((0,T)\times \Sigma,H_{\#}^1(Z)/\R))^3$ such that up to a subsequence it holds with $ v_0^M= (0,0,\partial_t u_0^3)^T$ that 
\begin{align*}
\chi_{\oemf}D(\veps^M) \rats \chi_{Z^f}\left( D_{\x} (v_0^M) + D_y(v_1^M)\right).
\end{align*}
\end{lemma}
\begin{proof}
We denote by $\tveps^M$ the extension from Lemma \ref{KornInequalityPerforatedLayer}, which fulfills the \textit{a priori} estimate (see also Lemma \ref{AprioriEstimatesLemma})
\begin{align*}
\frac{1}{\sqrt{\epsilon}} \Vert  \tveps^M\Vert_{L^2((0,T)\times \oem)} + \sqrt{\epsilon} \Vert \nabla \tveps^M\Vert_{L^2((0,T)\times \oem)} + \frac{1}{\sqrt{\epsilon}} \Vert D(\tveps^M)\Vert_{L^2((0,T)\times \oem)} \le C.
\end{align*}
From Proposition \ref{ConvergenceFluidMembraneProp} and Lemma \ref{LemmaBasicTSCompactness} we get the existence of $\xi \in L^2((0,T)\times \Sigma \times Z)$ such that up to a subsequence 
\begin{align*}
\tveps^M &\rats  v_0^M,
\\
D(\tveps^M) &\rats \xi.
\end{align*}
Let $\phi \in C_0^{\infty}((0,T)\times \Sigma)$ and $\psi \in L^2(Z)^{3 \times 3}$ symmetric with $\nabla_y \cdot \psi = 0$, and $Y$-periodic with $\psi \cdot \nu = 0$ on $S^{\pm}$, which means that for all $g \in H^1_{\#}(Z)^3 $ it holds that
\begin{align*}
\langle \psi \cdot \nu , g \rangle_{H^{-\frac12}(\partial Z), H^{\frac12}(\partial Z)} = 0.
\end{align*}
Then it holds with the integration by parts formula from \cite[Lemma 8]{GahnJaegerTwoScaleTools}
\begin{align*}
\int_0^T \int_{\Sigma} \int_Z \xi(t,\x,y) : \psi(y) \phi(t,\x) dy d\x dt &= \lim_{\epsilon\to 0} \foe  \int_0^T\int_{\oem} D(\tveps^M) : \psi\left(\fxe\right) \phi(t,\x) dx  dt
\\
&= \lim_{\epsilon\to 0} \foe \int_0^T \int_{\oem} \nabla \tveps^M : \psi\left(\fxe\right)\phi(t,\x) dx dt
\\
&= \lim_{\epsilon \to 0} -\foe \int_0^T \int_{\oem} \tveps^M \cdot \left[ \psi\left(\fxe\right)\nabla_{\x}\phi(t,\x) \right] dx dt
\\
&= -\int_0^T \int_{\Sigma} \int_Z   v_0^M(t,\x) \cdot \left[\psi(y)\nabla_{\x}\phi(t,\x) \right]dy d\x dt 
\\
&= \int_0^T \int_{\Sigma} \int_Z   D_{\x}(v_0^M)(t,\x) : \psi(y) \phi(t,\x) dy d\x dt.
\end{align*}
Due to the periodic Helmholtz-decomposition for symmetric matrix-valued functions \cite[Lemma 7]{GahnJaegerTwoScaleTools}, there exists $v_1^M \in L^2((0,T)\times \Sigma, H_{\#}^1(Z)/\R)^3$ such that
\begin{align*}
\xi =  D_{\x}(v_0^M) + D_y(v_1^M).
\end{align*}
This implies the desired result.
\end{proof}
Next, we show a continuity condition on the interface $\Gamma$ between the corrector $v_1^M$ and the velocity of the displacement.
\begin{lemma}\label{BCV1MLemma}
Let $v_1^M$ be the corrector from Lemma \ref{LemmaExistenceCorrectorV1M} and $\tilde{u}_1$ and $u_0^3$ the limit functions from Proposition \ref{ConvergenceResultsDisplacementProp}. Then we have 
\begin{align*}
v_1^M(t,\x,y) = \partial_t\tilde{u}_1(t,\x) - y_3 \nabla_{\x} \partial_t u_0^3(t,\x) \qquad\mbox{f.a.e. } (t,\x,y) \in (0,T)\times \Sigma \times \Gamma.
\end{align*}
\end{lemma}
\begin{proof}
For $\ast \in \{s,f\}$ and $f \in C_0^{\infty}(\Gamma)^3/\R^3$, let  $q^{\ast} \in H^1_{\#}(Z^{\ast})^3/\R^3$ be the unique weak solution of 
\begin{align}
\begin{aligned}
\label{AuxiliaryProblemCorrectorV1M}
\nabla \cdot (D(q^{\ast})) &= 0 &\mbox{ in }& Z^{\ast},
\\
D(q^{\ast})\nu^{\ast} &= f &\mbox{ on }& \Gamma,
\\
D(q^{\ast}) \nu^{\ast} &= 0 &\mbox{ on }& S^{\pm},
\\
q \mbox{ is } Y\mbox{-periodic}, \, &\int_{Z^{\ast}} q dy = 0,
\end{aligned}
\end{align}
where $\nu^{\ast}$ denotes the outer unit normal on $\partial Z^{\ast}$ with respect to $Z^{\ast}$.
We emphasize that for $\ast = s$ the condition for the normal trace on $S^{\pm}$ is not necessary, however, we see that the result is still valid if $Z^s$ touches $S^{\pm}$ in a nice way (see \cite{GrisvardEllipticProblems} for more details on this subject). Since the only $Y$-periodic rigid-displacements on $Z^{\ast}$ are the constant functions, the Korn-inequality in \cite[Chapter I, Theorem 2.5]{Oleinik1992} and the Lax-Milgram lemma implies the existence of a unique weak solution $q^{\ast} \in H^1_{\#}(Z^{\ast})^3/\R^3$. Since $f$ is smooth with compact support in $\Gamma$ and $\geps$ is $C^{1,1}$, the elliptic regularity theory, see for example \cite{GrisvardEllipticProblems},  implies $q^{\ast} \in H^2(Z^{\ast})^3$. 

Now, we define $\psi^{\ast}:= D(q^{\ast}) \in H^1(Z^{\ast})^{3\times 3}$, which has the following properties:
$\nabla_y \cdot \psi^{\ast} = 0$, $\psi^{\ast}$ is symmetric and $Y$-periodic,  $\psi^{\ast} \nu = f $ on $\Gamma$ and $\psi^{\ast}\nu = 0$ on $S^{\pm}$. Choosing $\phi \in C_0^{\infty}((0,T)\times \Sigma)$, we obtain with Lemma \ref{LemmaExistenceCorrectorV1M}
\begin{align*}
\lim_{\epsilon\to 0} \foe &\int_0^T \int_{\oemf} D(\veps^M): \psi^f\left(\fxe\right) \phi(t,\x) dx dt 
\\
&= \int_0^T \int_{\Sigma} \int_{Z^f} \left[D_{\x}(v_0^M)(t,\x) + D_y(v_1^M)(t,\x,y)\right] : \psi^f (y) \phi(t,\x) dy d\x dt.
\end{align*}
Integration by parts on the left-hand side gives with the continuity condition $\partial_t \ueps = \veps^M$ on $\geps$ and the two-scale convergence of $\veps^M$ from Proposition \ref{ConvergenceFluidMembraneProp}
\begin{align*}
 \foe &\int_0^T \int_{\oemf} D(\veps^M): \psi^f\left(\fxe\right) \phi(t,\x) dx dt 
\\
&= -\foe \int_0^T \int_{\oemf} \veps^M \cdot \left[\psi^f\left(\fxe\right) \nabla_{\x} \phi(t,\x)\right] dx dt + \foe \int_0^T \int_{\geps} \partial_t \ueps \cdot f\left(\fxe\right) \phi(t,\x) d\sigma dt
\\
&\overset{\epsilon \to 0}{\longrightarrow}  - \int_0^T \int_{\Sigma} \int_{Z^f} v_0^M \cdot \psi^f(y)\nabla_{\x}\phi(t,\x)dy d\x dt + \lim_{\epsilon\to 0}  \foe \int_0^T \int_{\geps} \partial_t \ueps \cdot f\left(\fxe\right) \phi(t,\x) d\sigma dt.
\end{align*}
For the boundary term we use, see Lemma \ref{MainTheoremTwoScaleConvergence} in the appendix,
\begin{align*}
\chi_{\oems} \nabla \partial_t \ueps \rats \chi_{Z^s} \left( \nabla_{\x} \partial_t u_0^3 + \nabla_y \partial_t u_1\right),
\end{align*}
with $u_1(t,\x,y) = \tilde{u}_1(t,\x) - y_3 \nabla_{\x} u_0^3(t,\x)$, to obtain with $\psi^s \nu = f$ on $\Gamma$
\begin{align*}
\foe &\int_0^T \int_{\geps} \partial_t \ueps \cdot f\left(\fxe\right) \phi(t,\x) d\sigma dt 
\\
&= \foe \int_0^T \int_{\oems} \partial_t \ueps \cdot \left[\psi^s\left(\fxe\right)\nabla_{\x} \phi(t,\x)\right]dx dt + \foe \int_0^T\int_{\oems} \nabla \partial_t \ueps : \psi^s \left(\fxe\right) \phi(t,\x) dx dt 
\\
&\overset{\epsilon \to 0}{\longrightarrow} \int_0^T \int_{\Sigma} \int_{Z^s} \partial_t u_0 \cdot \left[ \psi^s(y)\nabla_{\x}\phi(t,\x) \right] dy d\x dt + \int_0^T \int_{\Sigma} \int_{Z^s} \left[\nabla_{\x} \partial_t u_0 + \nabla_y \partial_t u_1\right]: \psi^s(y) \phi(t,\x) dy d\x dt
\\
&=  -\int_0^T \int_{\Sigma}\int_{\Gamma} \partial_t u_1 \cdot f(y) \phi(t,\x) d\sigma_y d\x dt.
\end{align*}
Altogether, we obtain (using the symmetry of $\psi^f$ and again $\nabla \cdot \psi^f= 0$)
\begin{align*}
-\int_0^T \int_{\Sigma}\int_{\Gamma} \partial_t u_1 \cdot f(y) \phi(t,\x) d\sigma_y d\x dt &= \int_0^T \int_{\Sigma} \int_{Z^f} D_y(v_1^M) : \psi^f(y)\phi(t,\x) dy d\x dt 
\\
&= \int_0^T \int_{\Sigma} \int_{\Gamma} v_1^M\cdot f(y) \phi(t,\x) d\sigma_y d\x dt.
\end{align*}
This implies $\partial_t u_1 = v_1^M + C(t,\x)$ for a "constant" depending on $(t,\x)$. However, since we have chosen $u_1$ and $v_1^M $ in such a way that it has mean value zero with respect to $y$, it holds that $C(t,\x)=0$. This implies the desired result.
\end{proof}
Now we are able to characterize the corrector term $v_1^M$ and also the two-scale limit of the pressure $\peps^M$.
\begin{proposition}\label{PropV1M}
It holds that
\begin{align*}
v_1^M &= \partial_t \tilde{u}_1 - y_3 \nabla_{\x} \partial_t u_0^3 &\mbox{ in }& (0,T)\times \Sigma \times Z^f,
\\
\nabla_y \cdot v_1^M &= 0 &\mbox{ in }& (0,T)\times \Sigma \times Z^f,
\end{align*}
and up to a subsequence we have
\begin{align*}
\chi_{\oemf}\peps^M \rats 0.
\end{align*}
\end{proposition}
\begin{proof}
First of all, denoting by $tr(A)$ the trace of a matrix $A$, we obtain from Lemma \ref{LemmaExistenceCorrectorV1M} 
\begin{align*}
0 = \nabla \cdot \veps^M = \mathrm{tr}(D(\veps^M)) \rats \mathrm{tr}\left(D_{\x}(v_0^M) + D_y(v_1^M)\right) = \nabla_y \cdot v_1^M.
\end{align*}
Hence, we have $\nabla_y \cdot v_1^M = 0$.
Due to the \textit{a priori} estimates in Lemma \ref{AprioriEstimatesLemma}, there exists $p_0^M \in L^2((0,T)\times \Sigma \times Z)$ such that up to a subsequence
\begin{align*}
\chi_{\oemf} \peps^M \rats \chi_{Z^f} p_0^M.
\end{align*}
Now, let $\phi \in C_0^{\infty}((0,T)\times \Sigma , C_{\#}^{\infty}(\overline{Z^f}))^3$ with compact support in $\overline{Z^f}\setminus \Gamma$, and $\rho \in C_0^{\infty}([1,2))$ such that $0 \le \rho \le 1$ and $\rho = 1$ in $\left[1,\frac32 \right]$. We define 
\begin{align*}
\phieps(t,x):= 
\begin{cases}
\epsilon \phi^M\left(t,\x,\dfrac{\x}{\epsilon} ,\pm 1\right) \rho\left(\pm \dfrac{x_n}{\epsilon}\right) &\mbox{ in } (0,T)\times \oepm,
\\
\epsilon \phi^M\left(t,\x,\fxe\right) &\mbox{ in } (0,T)\times \oemf,
\\
0 &\mbox{ in } (0,T)\times \oems.
\end{cases}
\end{align*}
We choose $\phieps $ as a test-function in $\eqref{EQ:Var_Micro}$. The terms in the solid domain are zero, since $\phieps = 0$ in $\oems$. Further, the terms in the bulk domains are of order $\sqrt{\epsilon}$, due to the cut off function $\rho$, see \cite[Proof of Theorem 5.2]{BhattacharyaGahnNeussRadu} for more details. Hence, for $\epsilon \to 0$ we get with the \textit{a priori} estimates from Lemma \ref{AprioriEstimatesLemma} and the convergence results in Proposition \ref{ConvergenceFluidMembraneProp} and Lemma \ref{LemmaExistenceCorrectorV1M}
\begin{align*}
\int_0^T \int_{\Sigma}\int_{Z^f} [D_{\x}(v_0^M) + D_y(v_1^M)] : D_y(\phi)  dy d\x dt  - \int_0^T \int_{\Sigma} \int_{Z^f} p_0^M \nabla_y \cdot \phi  dy d\x dt = 0.
\end{align*}
By density and using the boundary condition from Lemma \ref{BCV1MLemma} we obtain that $v_1^M$ is a weak solution of 
\begin{align*}
-\nabla_y \cdot \left( D_{\x}(v_0^M) + D_y(v_1^M)\right) + \nabla_y p_0^M &= 0 &\mbox{ in }& (0,T)\times \Sigma \times Z^f,
\\
\nabla_y \cdot v_1^M &= 0 &\mbox{ in }& (0,T)\times \Sigma \times Z^f,
\\
v_1^M &= \partial_t \tilde{u}_1 - y_3 \nabla_{\x} \partial_t u_0^3 &\mbox{ on }& (0,T)\times \Sigma \times \Gamma,
\\
\left(D_{\x}(v_0^M) + D_y(v_1^M) - p_0^M I\right)\nu &= 0 &\mbox{ on }& (0,T)\times \Sigma \times S^{\pm},
\\
v_1^M \mbox{ is } &Y\mbox{-periodic}, \, \int_{Z^f} v_1^M dy = 0.
\end{align*}
Using again the Korn-inequality in \cite[Chapter I, Theorem 2.5]{Oleinik1992}, the theory on Stokes equation implies that this problem has a unique weak solution $(v_1^M,p_0^M)$. It is easy to check that the function
\begin{align*}
(v_1^M,p_0^M) = (\partial_t \tilde{u}_1 - y_3 \nabla_{\x} \partial_t u_0^3 ,0)
\end{align*}
is a solution.
\end{proof}

\section{Conclusion}
\label{SectionConclusion}
In summary, we showed that in the topology of the two-scale convergence, the microscopic solution $(\veps,\peps,\ueps)$ can be approximated by 
\begin{align*}
v_{\epsilon,\app}^{\pm}(t,x) &= v_0^{\pm}(t,x \mp \epsilon e_n) &\mbox{ in }& (0,T)\times \oepm,
\\
p_{\epsilon,\app}^{\pm}(t,x) &= p_0^{\pm}(t,x \mp e_n) &\mbox{ in }& (0,T)\times \oepm,
\\
v_{\epsilon,\app}^M(t,x) &= \partial_t u_0^3(t,\bar{x}) e_3 + \epsilon \left[\partial_t \tilde{u}_1(t,\bar{x}) -\dfrac{x_3}{\epsilon} \nabla_{\x} \partial_t u_0^3(t,\bar{x}) \right] &\mbox{ in }& (0,T)\times \oemf,
\\
p_{\epsilon,\app}^M(t,x) &= 0 &\mbox{ in }& (0,T)\times \oemf,
\\
u_{\epsilon,\app}(t,x) &= u_0^3(t,\bar{x}) e_3 + \epsilon \left[ \tilde{u}_1(t,\bar{x}) - \dfrac{x_3}{\epsilon} \nabla_{\x} u_0^3(t,\bar{x}) \right] + \epsilon^2 u_2\left(t,\x,\fxe\right) &\mbox{ in }& (0,T)\times \oems.
\end{align*}
The approximate fluid velocity in the layer $v_{\epsilon,\app}^M$ is equal to the time derivative of the first two terms in the approximate displacement $u_{\epsilon,\app}$. In other words, in this order of approximation  the fluid does not transport substances transversal through the layer, Using a formal asymptotic expansion, we expect that the second order-corrector for the fluid velocity differs from $\epsilon^2 \partial_t u_2$, but a rigorous proof is missing.
 The transversal flux through the porous layer is important in applications, even if it is small, since such small effects may sum up and have a relevant impact in the long time.
This is the case, for example, in physiological processes where exchange through endothelial and epithelial layers between adjacent compartments can occur by paracellular or transcellular diffusion, and also by paracellular transport in fluid.
Therefore, determining higher order corrector terms is one of the topics of ongoing research. Likewise,  the linearization of the kinetic relation and the assumption of small deformations have to be eliminated and deserve special attention.

\section*{Acknowledgement(s)}

This research contributes to the mathematical modeling  of inflammation as an immune response to infections 
and is supported by SCIDATOS (Scientific Computing for Improved Detection and Therapy of Sepsis). 
SCIDATOS is a collaborative project funded by the Klaus Tschira Foundation, Germany (Grant Number 00.0277.2015) 
and provided in particular the funding for the  research of the first author.

\bibliographystyle{abbrv}
\bibliography{literature}

\appendix

\section{Auxiliary results}
\label{SectionAppendixAuxiliary}
In this section we recall some technical results. We start with a Korn-inequality for perforated thin layers \cite[Theorem 2]{GahnJaegerTwoScaleTools}:
\begin{lemma}\label{KornInequalityPerforatedLayer}
For all $\weps \in H^1(\Omega_{\epsilon}^{M,\ast})^3$ for $\ast \in \{s,f\}$ with $\weps = 0 $ on $ \partial_D \Omega_{\epsilon}^{M,\ast}$ it holds that 
\begin{align*}
\sum_{i=1}^2 \foe \Vert \weps^i\Vert_{L^2(\Omega_{\epsilon}^{M,\ast})} + \sum_{i,j=2}^2 \foe\Vert \partial_i \weps^j \Vert_{L^2(\Omega_{\epsilon}^{M,\ast})} + \Vert \weps^3\Vert_{L^2(\Omega_{\epsilon}^{M,\ast})} +  \Vert \nabla \weps \Vert_{L^2(\Omega_{\epsilon}^{M,\ast})} \le \frac{C}{\epsilon} \Vert D(\weps)\Vert_{L^2(\Omega_{\epsilon}^{M,\ast})}.
\end{align*}
\end{lemma}
Further we use the following extension operator which in particular preserves the uniform \textit{a priori} bound for the symmetric gradient \cite[Theorem 1]{GahnJaegerTwoScaleTools}:
\begin{lemma}\label{TheoremExtensionOperator}
There exists an extension operator $E_{\epsilon} : H^1(\Omega_{\epsilon}^{M,\ast})^n \rightarrow H^1(\oem)^3$ for $\ast\in \{s,f\}$, such that for all $\weps \in H^1(\Omega_{\epsilon}^{M,\ast})^3$ it holds that ($i=1,2,3$)
\begin{align*}
\Vert (E_{\epsilon}\weps)^i \Vert_{L^2(\oem)} &\le C \left( \Vert \weps^i \Vert_{L^2(\Omega_{\epsilon}^{M,\ast})} + \epsilon \Vert \nabla \weps \Vert_{L^2(\Omega_{\epsilon}^{M,\ast})} \right),
\\
\Vert \nabla E_{\epsilon}\weps \Vert_{L^2(\oem)} &\le  C \Vert \nabla \weps \Vert_{L^2(\Omega_{\epsilon}^{M,\ast})},
\\
\Vert D(E_{\epsilon}\weps) \Vert_{L^2(\oem)} &\le C \Vert D(\weps)\Vert_{L^2(\Omega_{\epsilon}^{M,\ast})},
\end{align*}
for a constant $C\gr 0$ independent of $\epsilon$. 
\end{lemma}

\section{Two-scale convergence}
\label{AppendixTwoScaleConvergence}

We briefly introduce two-scale convergence concepts for thin layers \cite{BhattacharyaGahnNeussRadu,GahnNeussRadu2017EffectiveTransmissionConditions,NeussJaeger_EffectiveTransmission}, and recall the compactness results used in this paper.
\begin{definition}\
\begin{enumerate}
[label = (\roman*)]
\item\,[Two-scale convergence in the thin layer $\oem$] We say the sequence $\weps \in L^2((0,T)\times \oem)$ converges (weakly) in the two-scale sense to a limit function $w_0 \in L^2((0,T)\times  \Sigma \times Z)$ if 
\begin{align*}
\lim_{\epsilon\to 0} \foe \int_0^T \int_{\oem} \weps(t,x) \phi \tbxfxe dxdt = \int_0^T\int_{\Sigma} \int_Z w_0(t,\x,y) \psi(t,\x,y) dy d\x dt
\end{align*}
for all $\phi \in L^2((0,T)\times \Sigma,C_{\#}^0(\overline{Z}))$. We write 
\begin{align*}
\weps \rats w_0.
\end{align*}
\item\,[Two-scale convergence on the oscillating surface $\geps$] We say the sequence $\weps \in L^2((0,T)\times \geps)$ converges (weakly) in the two-scale sense to a limit function $w_0 \in L^2((0,T)\times  \Sigma \times \Gamma)$ if 
\begin{align*}
\lim_{\epsilon\to 0}  \int_0^T\int_{\geps} \weps(t,x) \phi \tbxfxe dx dt = \int_0^T \int_{\Sigma} \int_\Gamma w_0(t,\x,y) \psi(t,\x,y) dy d\x dt
\end{align*}
for all $\phi \in C^0([0,T]\times \overline{\Sigma},C_{\#}^0(\Gamma))$. We write  
\begin{align*}
\weps \rats w_0 \qquad \mbox{on } \geps.
\end{align*}
\end{enumerate}
\end{definition}
The following lemma gives basic compactness results for the two-scale convergence in thin layers. 
\begin{lemma}\label{LemmaBasicTSCompactness}\
\begin{enumerate}
[label = (\roman*)]
\item Let $\weps \in L^2((0,T), H^1(\oem))$ be a sequence with
\begin{align*}
\frac{1}{\sqrt{\epsilon}}\Vert \weps \Vert_{L^2((0,T)\times \oem)} + \sqrt{\epsilon}\Vert \nabla \weps \Vert_{L^2((0,T)\times \oem)}  \le C.
\end{align*}
Then there exists a subsequence (again denoted $\weps$) and a limit function $w_0 \in  L^2( (0,T)\times \Sigma, H_{\#}^1(Z)/\R)^3$ such that the following two-scale convergences hold
\begin{align*}
\weps  &\rats w_0,
\\
\nabla \weps &\rats \nabla_y w_0
\end{align*}
\item  Consider the sequence $\weps \in L^2((0,T)\times \geps)$ with
\begin{align*}
\Vert \weps \Vert_{L^2((0,T)\times \geps)} \le C.
\end{align*}
Then there exists a subsequence (again denoted $\weps$) and a limit function $w_0 \in  L^2( (0,T)\times \Sigma\times \Gamma)$ such that  
\begin{align*}
\weps \rats w_0 \qquad \mbox{on } \geps.
\end{align*}
\end{enumerate}
\end{lemma}
We close this section with the following rather recent compactness result with respect to two-scale convergence for sequences of vector valued functions defined on thin perforated layers, describing e.g., the displacement of the layer. The two-scale limit represents a Kirchhoff-Love displacement. A proof is given in \cite{GahnJaegerTwoScaleTools}, and similar results in the framework of the unfolding operator and a slightly different condition at the outer boundary can be found in \cite{griso2020homogenization}.
\begin{lemma}\label{MainTheoremTwoScaleConvergence}
Let $\weps\in L^2((0,T),H^1(\oems))^3$ with $\weps = 0$ on $\partial_D \oems$ be a sequence with
\begin{align*}
\Vert \weps^3 \Vert_{L^2((0,T)\times\oems)} + \Vert \nabla \weps \Vert_{L^2((0,T)\times\oems)} + \foe\Vert D(\weps)\Vert_{L^2((0,T)\times\oems)} + \sum_{\alpha =1}^2 \foe\Vert \weps^{\alpha} \Vert_{L^2((0,T)\times\oems)} \le C\sqrt{\epsilon}.
\end{align*}
Then  there exist $w_0^3 \in L^2((0,T),H^2_0(\Sigma))$, $\tilde{w}_1 \in L^2((0,T),H^1_0(\Sigma))^3$ with $\tilde{w}_1^3 = 0$, and $w_2 \in L^2( (0,T)\times \Sigma, H_{\#}^1(Z)/\R)^3$ such that up to a subsequence (for $\alpha = 1,2$)
\begin{align*}
\chi_{\oems}\weps^3 &\rats \chi_{Z^s}w_0^3,
\\
\chi_{\oems}\frac{\weps^{\alpha}}{\epsilon} &\rats \chi_{Z^s}\big(\tilde{w}_1^{\alpha} - y_3 \partial_{\alpha} w_0^3\big),
\\
\frac{1}{\epsilon} \chi_{\oems} D(\weps) &\rats  \chi_{Z^s} \left(D_{\x}(\tilde{w}_1) - y_3 \nabla_{\x}^2 w_0^3 + D_y(w_2) \right),
\\
\chi_{\oems} \nabla \weps &\rats \chi_{Z^s} \left[ \nabla_{\x} (0,0,w_0^3)^T + \nabla_y \big(\tilde{w}_1^{\alpha} - y_3 \partial_{\alpha} w_0^3\big) \right].
\end{align*}
Further, the function $\tilde{w}_1 - y_3 \nabla_{\x} w_0^3$ has mean value zero in $Z^s$ for almost every $(t,\x) \in (0,T)\times \Sigma$.
\end{lemma}
\begin{proof}
See \cite[Theorem 3]{GahnJaegerTwoScaleTools} and for the convergence of the gradient \cite[Proof of Proposition 2]{GahnJaegerTwoScaleTools}.
\end{proof}

\end{document}